\documentclass[preprint,12pt,times,numbers]{elsarticle}

\usepackage{amssymb}
\usepackage{subfigure}
\usepackage{mathrsfs,amsbsy}
\usepackage{dsfont}
\usepackage{graphicx}
\usepackage{amsmath,amsfonts,amsthm}
\usepackage{latexsym}
\usepackage{color}
\usepackage{fixmath}
\usepackage{version}
\excludeversion{hide}

\newtheorem{theorem}{Theorem}[section]
\newtheorem{lemma}{Lemma}[section]
\newtheorem{corollary}{Corollary}[section]
\newtheorem{prop}{Proposition}[section]
\newtheorem{remark}{{\sc Remark}}[section]

\newcommand{\diag}{\operatorname{diag}}
\newcommand{\half}{\frac{1}{2}}
\newcommand{\ignore}[1]{}

\newcommand{\bdm}{\begin{displaymath}}
\newcommand{\edm}{\end{displaymath}}
\renewcommand{\(}{\left(}
\renewcommand{\)}{\right)}
\newcommand{\cond}{\,|\,}
\newcommand{\Sm}[2]{{\displaystyle \sum_{#1}^{#2}}}

\newcommand{\bbe}{{\boldsymbol e}}
\newcommand{\bq}{{\boldsymbol q}}
\newcommand{\bv}{{\boldsymbol v}}
\newcommand{\bx}{{\boldsymbol x}}
\newcommand{\by}{{\boldsymbol y}}
\newcommand{\bphi}{{\boldsymbol \phi}}

\newcommand{\C}{{\mathbb C}}
\newcommand{\N}{{\mathbb N}}
\newcommand{\Rf}{{\mathbb R}}
\newcommand{\Z}{{\mathbb Z}}
\newcommand{\D}{{\mathcal D}}

\journal{}
\begin{document}

\begin{frontmatter}

\title{Mysteries around the graph {L}aplacian eigenvalue 4}

\author[yuji]{Yuji Nakatsukasa}
\ead{yuji.nakatsukasa@manchester.ac.uk}

\author[ucd]{Naoki Saito\corref{ns}}
\ead{saito@math.ucdavis.edu}

\author[ucd]{Ernest Woei}
\ead{woei@math.ucdavis.edu}

\address[yuji]{School of Mathematics, The University of Manchester, Manchester, M13 9PL, UK}

\address[ucd]{Department of Mathematics, University of California, Davis, CA 95616, USA}

\cortext[ns]{Corresponding author}

%\baselineskip = 0.3cm

%\input titlepage.tex

% \maketitle

%\tableofcontents

%\clearpage
\begin{abstract}
We describe our current understanding on the phase transition phenomenon
of the graph Laplacian eigenvectors constructed on a certain type of 
unweighted trees, which we previously observed through our numerical experiments.
The eigenvalue distribution for such a tree is a smooth bell-shaped curve 
starting from the eigenvalue 0 up to 4.  Then, at the eigenvalue 4, 
there is a sudden jump.  Interestingly, the eigenvectors corresponding to 
the eigenvalues below 4 are \emph{semi-global} oscillations (like Fourier 
modes) over the entire tree or one of the branches; on the other
hand, those corresponding to the eigenvalues above 4 are much more 
\emph{localized} and \emph{concentrated} (like wavelets) around 
junctions/branching vertices.
For a special class of trees called \emph{starlike trees}, we obtain
a complete understanding of such phase transition phenomenon.
For a general graph, we prove the number of the eigenvalues larger than 4 
is bounded from above by the number of vertices whose degrees is strictly
higher than 2.
Moreover, we also prove that if a graph contains a branching path, then 
the magnitudes of the components of any eigenvector corresponding to the
eigenvalue greater than 4 decay exponentially from the branching vertex toward 
the leaf of that branch.
%\rr{remove: We have also identified a unique class of trees that can have an eigenvalue exactly equal to 4.}
\end{abstract}


\end{frontmatter}

\section{Introduction}
In our previous report \cite{SAITO-WOEI-DENDRITES}, we proposed a 
method to characterize dendrites of neurons, more specifically retinal ganglion
cells (RGCs) of a mouse, and cluster them into different cell types using their 
morphological features, which are derived from the eigenvalues
of the graph Laplacians when such dendrites are represented as graphs
(in fact literally as ``trees'').
For the details on the data acquisition and the conversion of dendrites
to graphs, see \cite{SAITO-WOEI-DENDRITES} and the references therein.
While analyzing the eigenvalues and eigenvectors of those graph
Laplacians, we observed a very peculiar \emph{phase transition phenomenon} 
as shown in Figure~\ref{fig:eigvaldist}.
\begin{figure}
\begin{center}
\renewcommand{\subfigtopskip}{0pt}
\renewcommand{\subfigbottomskip}{0pt}
\renewcommand{\subfigcapskip}{0pt}
\renewcommand{\subfigcapmargin}{0pt}
\renewcommand{\thesubfigure}{(a)}
\subfigure[RGC \#60]{\includegraphics[width=0.49\textwidth]{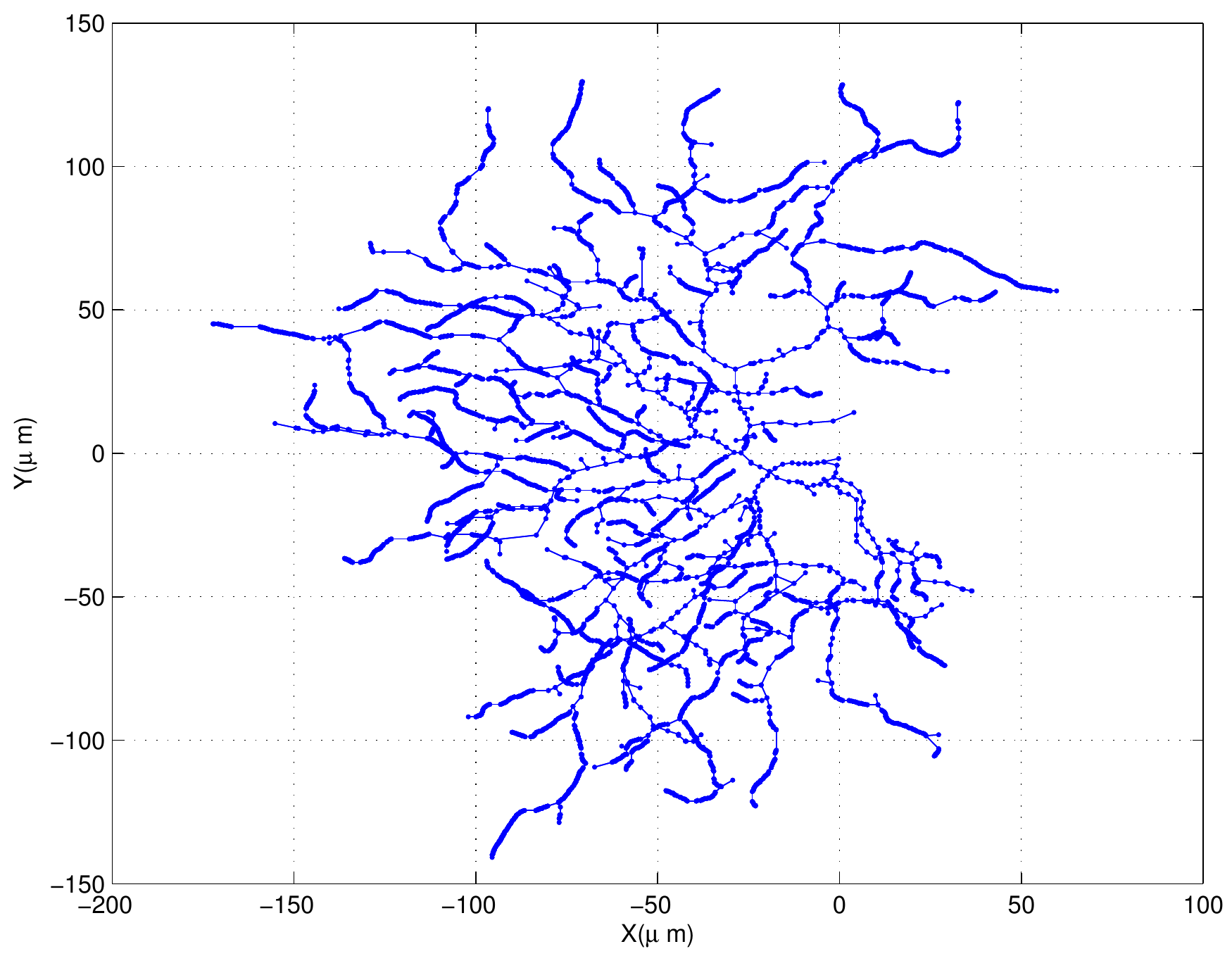}}
\renewcommand{\thesubfigure}{(b)}
\subfigure[RGC \#100]{\includegraphics[width=0.49\textwidth]{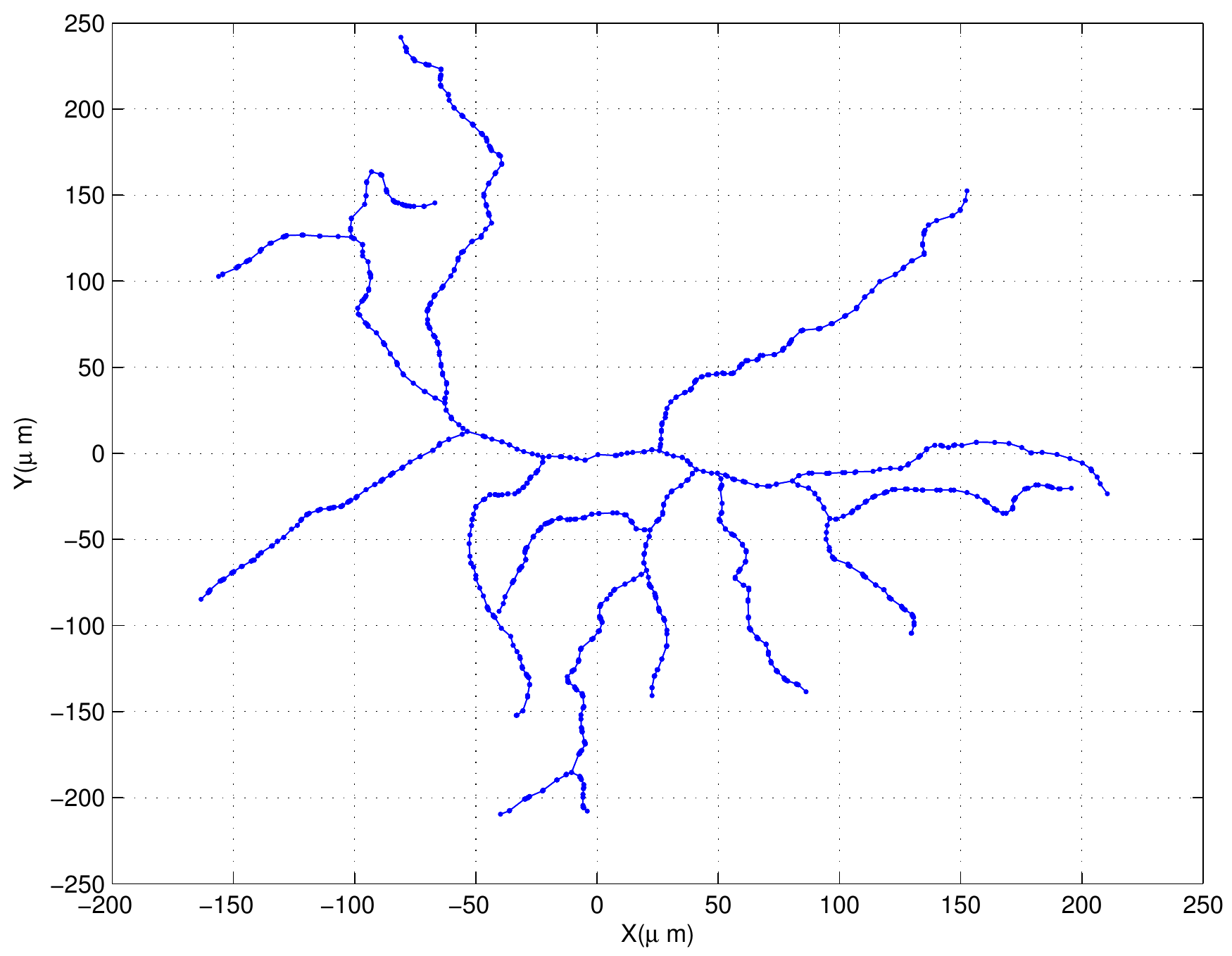}}\\
\renewcommand{\thesubfigure}{(c)}
\subfigure[Eigenvalues of (a)]{\includegraphics[width=0.49\textwidth]{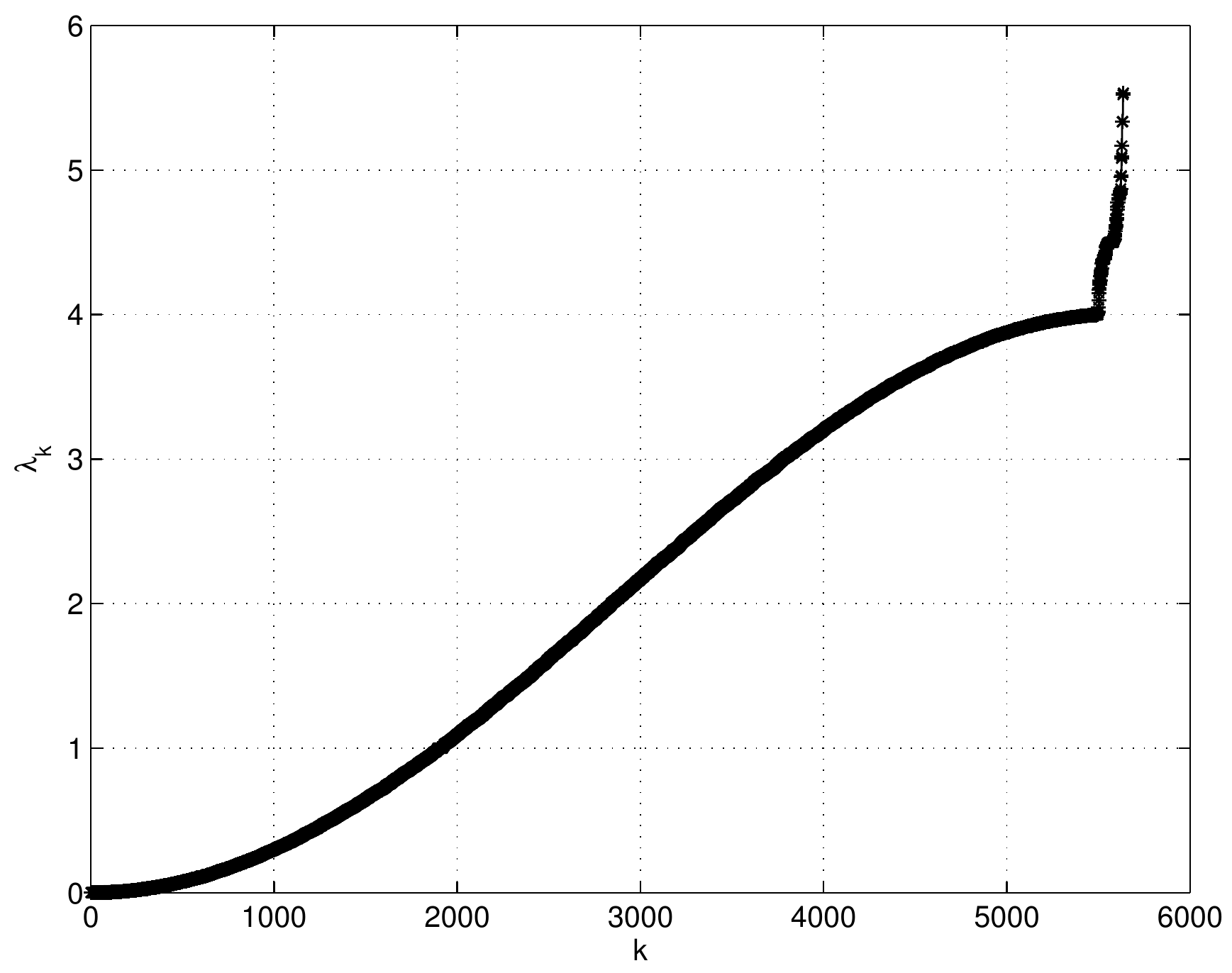}}
\renewcommand{\thesubfigure}{(d)}
\subfigure[Eigenvalues of (b)]{\includegraphics[width=0.49\textwidth]{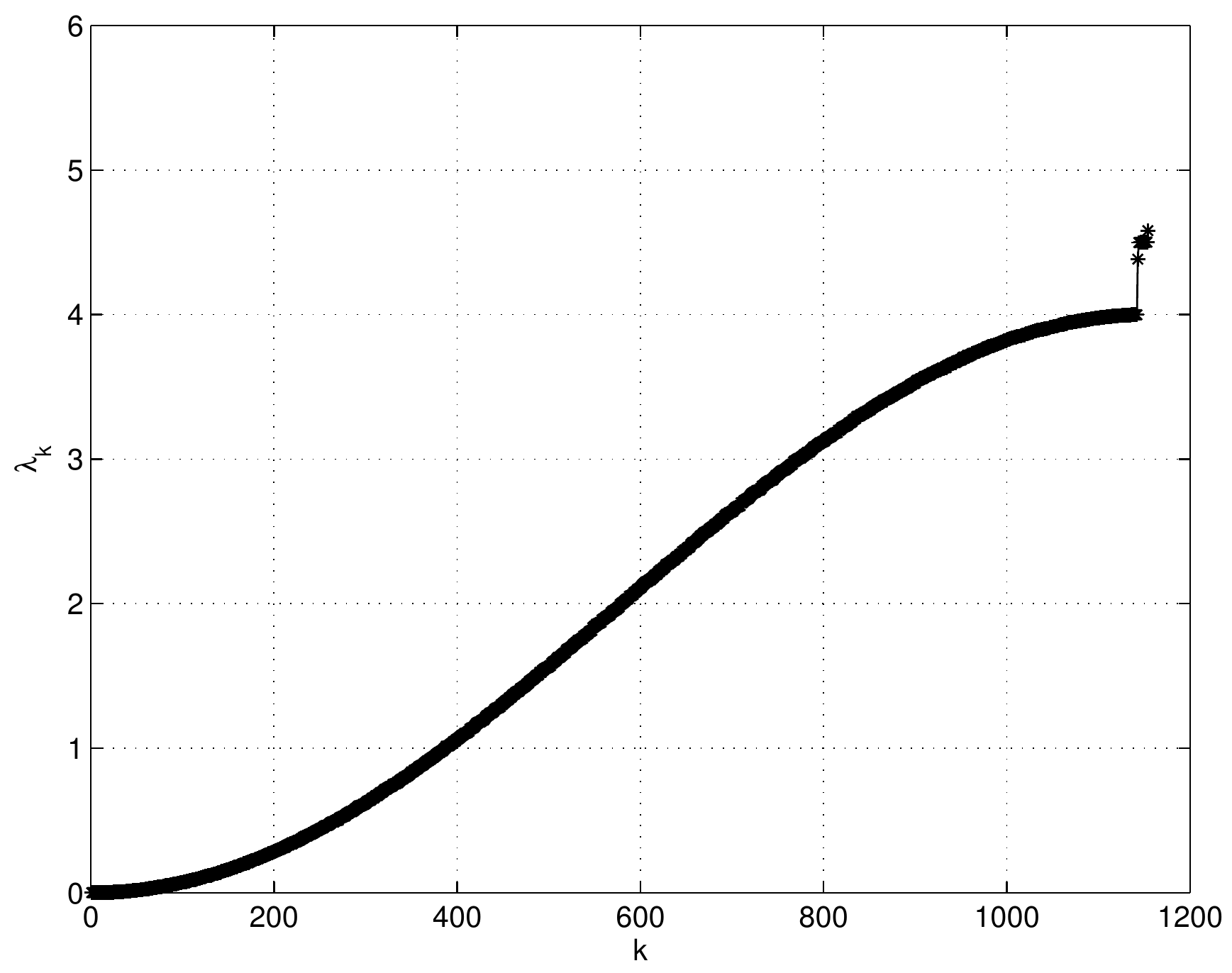}}
\end{center}
\caption{Typical dendrites of Retinal Ganglion Cells (RGCs) of a mouse
and the graph Laplacian eigenvalue distributions. 
(a) 2D projection of dendrites of an RGC of a mouse;
(b) that of another RGC revealing different morphology; 
(c) the eigenvalue distribution of the RGC shown in (a); 
(d) that of the RGC shown in (b).  
Regardless of their morphological features, a phase transition occurs at the 
eigenvalue 4.}
\label{fig:eigvaldist}
\end{figure}
The eigenvalue distribution for each dendritic tree
is a smooth bell-shaped curve starting from the eigenvalue 0 up to 4.
Then, at the eigenvalue 4, there is a sudden jump as shown in Figure~\ref{fig:eigvaldist}(c, d).  
Interestingly, the eigenvectors corresponding to the eigenvalues 
below 4 are \emph{semi-global} oscillations (like Fourier cosines/sines) over
the entire dendrites or one of the dendrite arbors (or branches); on the other
hand, those corresponding to the eigenvalues above 4 are much more 
\emph{localized} and \emph{concentrated} (like wavelets) around 
junctions/branching vertices, as shown in Figure~\ref{fig:eigvecs}.
\begin{figure}
\begin{center}
\renewcommand{\subfigtopskip}{0pt}
\renewcommand{\subfigbottomskip}{0pt}
\renewcommand{\subfigcapskip}{0pt}
\renewcommand{\subfigcapmargin}{0pt}
\renewcommand{\thesubfigure}{(a)}
\subfigure[RGC \#100; $\lambda_{1141}=3.9994$]{\includegraphics[width=0.49\textwidth]{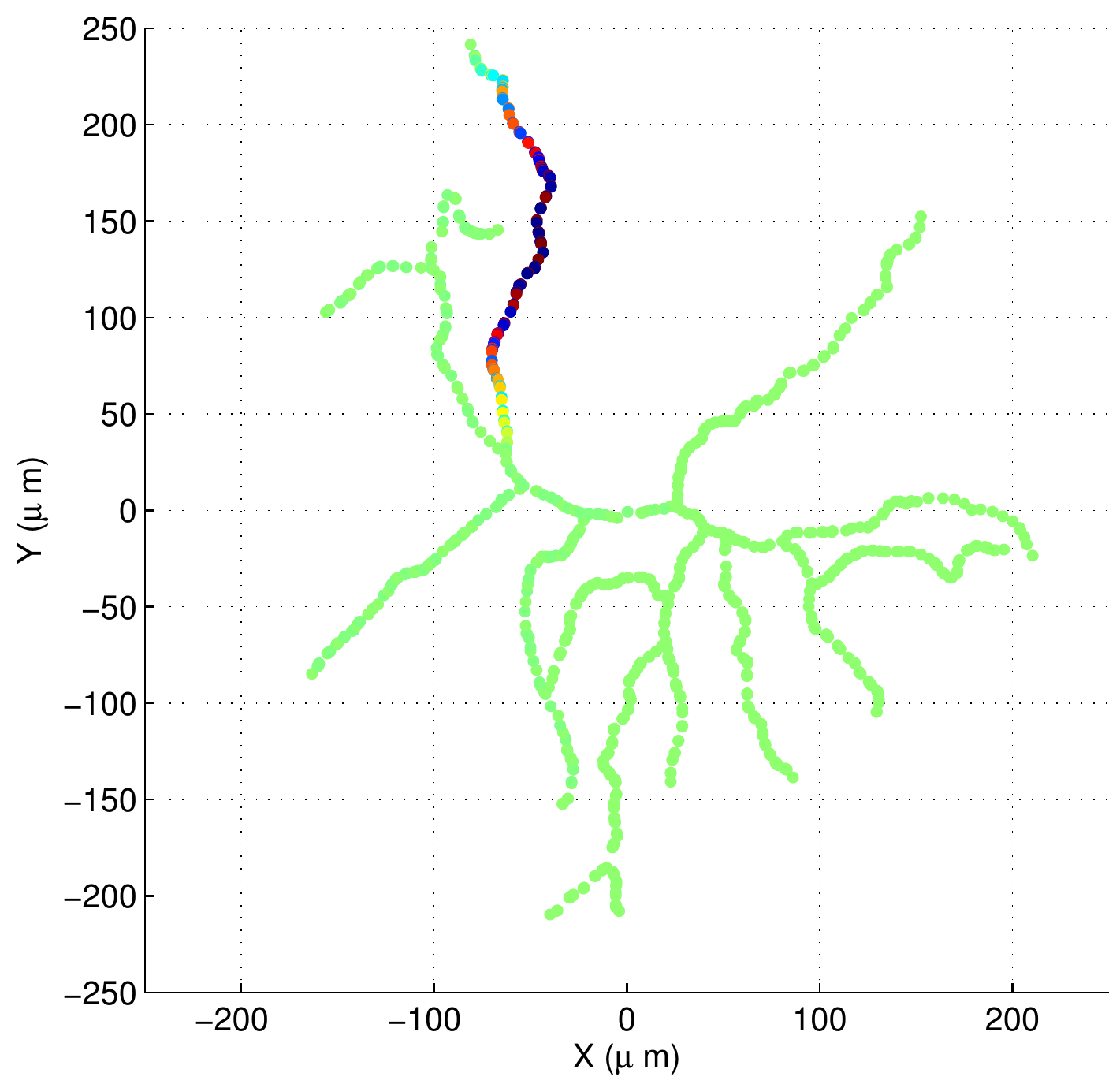}}
\renewcommand{\thesubfigure}{(b)}
\subfigure[RGC \#100; $\lambda_{1142}=4.3829$]{\includegraphics[width=0.49\textwidth]{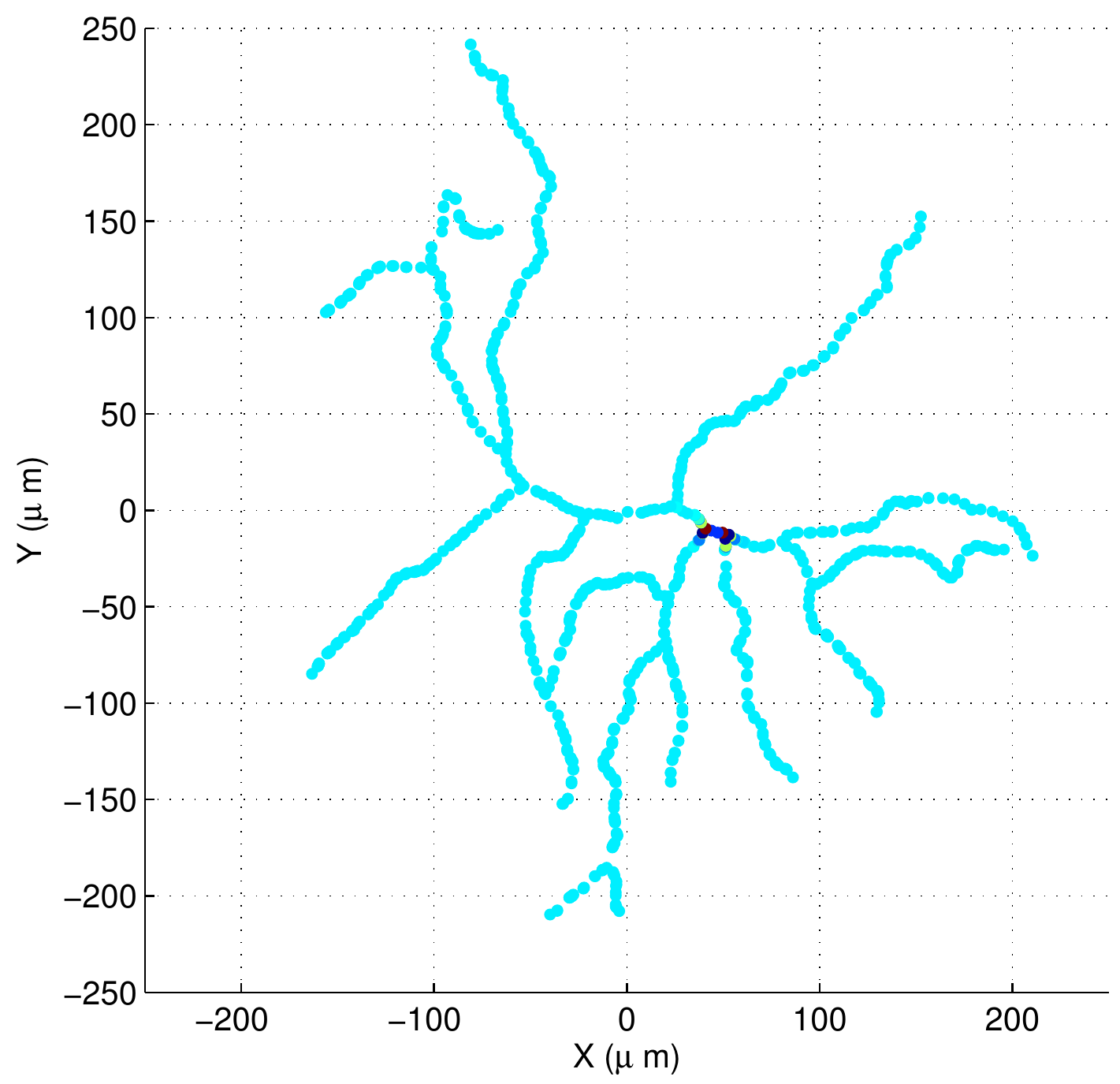}}
\end{center}
\caption{The graph Laplacian eigenvectors of RGC \#100.
(a) The one corresponding to the eigenvalue $\lambda_{1141}=3.9994$, immediately
below the value 4; (b) the one corresponding to the eigenvalue 
$\lambda_{1142}=4.3829$, immediately above the value 4.}
\label{fig:eigvecs}
\end{figure}

We want to answer the following questions:
\begin{description}
\item[Q1] Why does such a phase transition phenomenon occur?
\item[Q2] What is the significance of the eigenvalue 4?
\item[Q3] Is there any tree that possesses an eigenvalue exactly equal to 4?
\item[Q4] What about more general graphs that possess eigenvalues exactly equal to 4?
\end{description}
%\rr{remove: At this point of time, we have a complete answer to Q3, which will be described in Section~\ref{sec:mK13}.  }
As for Q1 and Q2, which are closely related, we have a complete answer for 
a specific and simple class of trees called \emph{starlike trees} as described
in Section~\ref{sec:starlike}, and a partial answer for more general trees and
graphs such as those representing neuronal dendrites, which we discuss in 
Section~\ref{sec:genres}.
For Q3, we identify two classes of trees that have an eigenvalue exactly equal
to 4, which is necessarily a simple eigenvalue, in Section~\ref{sec:mK13}.

In Section~\ref{sec:eigconsequence}, we will prove that the existence of a long
path between two subgraphs implies that the eigenvalues of either of the 
subgraphs that are larger than 4 are actually very close to some eigenvalues of
the whole graph.
Then, in Section~\ref{sec:counterex}, we will give a counterexample to the
conjecture that the largest component in the eigenvector corresponding to 
the largest eigenvalue (which is larger than 4) lies on the vertex of the 
highest degree.
Finally, we describe our investigation on Q4 in Section~\ref{sec:disc}.
But let us first start by fixing our notation and reviewing the basics of
graph Laplacians in Section~\ref{sec:defs}.

\section{Definitions and Notation}
\label{sec:defs}
Let $G=(V,E)$ be a graph where $V = V(G) = \{ v_1, v_2, \ldots, v_n \}$ is a set
of vertices in $G$ and $E = E(G) = \{ e_1, e_2, \ldots, e_m \}$ is a set of edges
where $e_k$ connects two vertices $v_i, v_j$ for some $1 \leq i,j \leq n$, and
we write $e_k = (v_i, v_j)$.  Let $d_k = d(v_k)$ be the degree of the vertex 
$v_k$.  If a graph $G$ is  a \emph{tree}, i.e.,
a connected graph without cycles, then it has $m=n-1$ edges.
Let $L(G) := D(G)-A(G)$ be the \emph{Laplacian matrix} where 
$D(G) := \diag(d_1, \ldots, d_n)$ 
is called the \emph{degree matrix} of $G$, i.e., the diagonal matrix of vertex 
degrees, and $A(G) = (a_{ij})$ is the \emph{adjacency matrix} of $G$, i.e., 
$a_{ij} = 1$ if $v_i$ and $v_j$ are adjacent; otherwise it is 0.
Furthermore, 
let $0=\lambda_0(G) \leq \lambda_1(G) \leq \cdots \leq \lambda_{n-1}(G)$ 
be the eigenvalues of $L(G)$, and
 $m_G(\lambda)$ be the multiplicity of the eigenvalue $\lambda$.
More generally, if $I \subset \Rf$ is an interval of the real line,
then we define $m_G(I) := \# \{ \lambda_k (G) \in I \}$.

At this point we would like to give a simple yet important example of
a tree and its graph Laplacian: a \emph{path} graph $P_n$ consisting of 
$n$ vertices shown in Figure~\ref{fig:path}.
\begin{figure}
\begin{center}
\includegraphics[width=0.8\textwidth]{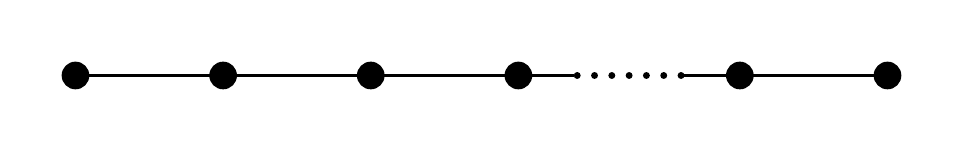}
\end{center}
\caption{A path graph $P_n$ provides a simple yet important example.}
\label{fig:path}
\end{figure}
The graph Laplacian of such a path graph can be easily obtained and
is instructive.  
%\begin{multline*}
\[
\underbrace{\begin{bmatrix}
 1 &     -1 &        &        &    \\
-1 &      2 &   -1   &        &    \\
   & \ddots & \ddots & \ddots &    \\
   &        &    -1  &    2   & -1 \\
   &        &        &   -1   &  1
\end{bmatrix}}_{L(G)} 
= \underbrace{\begin{bmatrix}
 1 &    &        &        &    \\
   &  2 &        &        &    \\
   &    & \ddots &        &    \\
   &    &        &    2   &    \\
   &    &        &        &  1
\end{bmatrix}}_{D(G)} 
-
\underbrace{\begin{bmatrix}
 0 &       1 &        &        &    \\
 1 &       0 &    1   &        &    \\
   &  \ddots & \ddots & \ddots &    \\
   &         &     1  &    0   &  1 \\
   &         &        &    1   &  0
\end{bmatrix}}_{A(G)} .
\]
%\end{multline*}
The eigenvectors of this matrix are nothing but the \emph{DCT Type II} basis
vectors used for the JPEG image compression standard; 
see e.g., \cite{STRANG-DCT}.  In fact, 
we have
\begin{eqnarray}
\lambda_k &=& 4 \sin^2\(\frac{\pi k}{2n}\) ; \\ \label{eqn:1d-ew}
%\bphi_k &=& \( \cos\(\frac{\pi k}{n}\(j+\half\)\) \)^{\mathsf{T}}_{0 \leq j < n}, \label{eqn:1d-ev}
%\bphi_k &=& \( \cos\(\frac{\pi k}{n}\(j-\half\)\) \)^{\mathsf{T}}_{1 \leq j \leq n}, \label{eqn:1d-ev}
\phi_{j,k} &=&  \cos\(\frac{\pi k}{n}\(j-\half\)\) , \quad 
 j=1,\ldots,n
\label{eqn:1d-ev}
\end{eqnarray}
for $k=0, 1, \ldots, n-1$, where 
%\bphi_{k,j}
$\bphi_{k} = \(\phi_{1,k}, \cdots, \phi_{n,k}\)^\mathsf{T}$ is the 
eigenvector corresponding to $\lambda_k$. 
From these, it is clear that for any finite $n \in \N$, 
$\lambda_{\max} = \lambda_{n-1} < 4$,
and no localization/concentration occurs in the eigenvector $\bphi_{n-1}$ 
(or any eigenvector), which is simply a global oscillation with the highest
possible (i.e., the Nyquist) frequency, i.e.,
$\bphi_{n-1} = \( (-1)^{j-1} \sin\(\frac{\pi}{n}\(j-\half\)\) \)^{\mathsf{T}}_{1 \leq j \leq n}$. 

\section{Analysis of Starlike Trees}
\label{sec:starlike}
As one can imagine, analyzing this phase transition phenomenon for complicated
dendritic trees turns out to be rather formidable.  Hence, we start our analysis
on a simpler class of trees called \emph{starlike trees}.
A starlike tree is a tree that has exactly one vertex of degree 
higher than 2.  Examples are shown in Figure~\ref{fig:starlike}.

\begin{figure}
\begin{center}
\renewcommand{\subfigtopskip}{0pt}
\renewcommand{\subfigbottomskip}{0pt}
\renewcommand{\subfigcapskip}{0pt}
\renewcommand{\subfigcapmargin}{0pt}
\renewcommand{\thesubfigure}{(a)}
\subfigure[$S(2,2,1,1,1,1)$]{\includegraphics[width=0.25\textwidth]{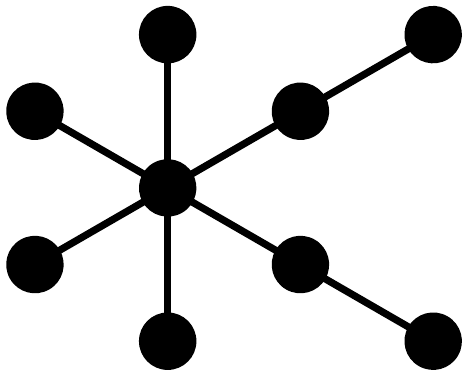}}
\hspace{2em}
\renewcommand{\thesubfigure}{(b)}
\subfigure[$S(n_1,1,1,1,1,1,1,1)$ a.k.a.\ comet]{\includegraphics[width=0.65\textwidth]{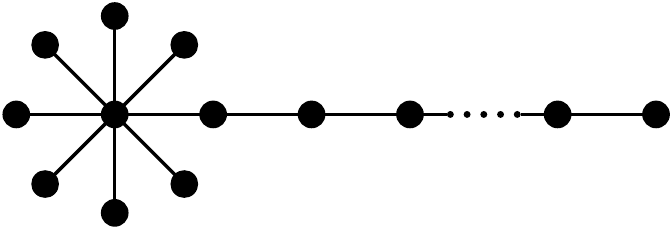}}
\end{center}
\caption{Typical examples of a starlike tree.}
\label{fig:starlike}
\end{figure}

We use the following notation.
Let $S(n_1, n_2, \dots, n_k)$ be a starlike tree that has $k (\geq 3)$ 
paths (i.e., branches) emanating from the central vertex $v_1$.
Let the $i$th branch have $n_i$ vertices excluding $v_1$.
Let $n_1 \geq n_2 \geq \cdots \geq n_k$.
Hence, the total number of vertices is $n=1 + \Sm{i=1}{k} n_i$.

Das proved the following results for a starlike tree $S(n_1, \ldots, n_k)$
in \cite{DAS}:
\bdm
\lambda_{\max} = \lambda_{n-1} < k+1+\frac{1}{k-1} ;
\edm
\begin{equation}
\label{eqn:2ndeigval}
2 + 2\cos\( \frac{2\pi}{2n_k+1}\) \leq \lambda_{n-2} \leq 2 + 2\cos\( \frac{2\pi}{2n_1+1}\) .
\end{equation}
On the other hand, Grone and Merris \cite{GRONE-MERRIS-2} proved the following 
lower bound for a general graph $G$ with at least one edge:
\begin{equation}
\label{eqn:lowerbound}
\lambda_{\max} \geq \max_{1 \leq j \leq n} d(v_j) + 1.
\end{equation}
Hence we have the following
\begin{corollary}
\label{cor:starlike}
A starlike tree has exactly one graph Laplacian
eigenvalue greater than or equal to 4.  The equality holds if and only if
the starlike tree is $K_{1,3}=S(1,1,1)$, which is also known as a \emph{claw}.
\end{corollary}
\begin{proof}
The first statement is easy to show.
The lower bound in \eqref{eqn:lowerbound} is larger than or equal to 4 
for any starlike tree since $\max_{1 \leq j \leq n} d(v_j) = d(v_1) \geq 3$.
On the other hand, the second largest eigenvalue $\lambda_{n-2}$ 
%clearly cannot exceed 4 due to \eqref{eqn:2ndeigval}.
is clearly strictly smaller than 4 due to \eqref{eqn:2ndeigval}.

\begin{hide}
To prove the second statement about the necessary condition on the 
equality (the sufficiency is easily verified), first note that by 
\eqref{eqn:lowerbound}, $d(v_1)=3$ is a necessary condition for 
$\lambda_{\max}=4$. Then %there exists a permutation such that
 the Laplacian matrix of the starlike tree can be written as
$
\begin{bmatrix}
  L_{11}&  L_{21}^{\mathsf{T}}\\
  L_{21}&  L_{22}
\end{bmatrix}$, where 
$
L_{11}=\mbox{\footnotesize $\begin{bmatrix}
  3&-1&-1&-1\\
  -1&1+a_1&0&0\\
  -1&0&1+a_2&0\\
  -1&0&0&1+a_3\\
\end{bmatrix}$}$ and
 $a_1,a_2,a_3$ are nonnegative integers ($L_{21},L_{22}$ are possibly empty). 
%unless the starlike tree is the claw $K_{1,3}$. 
Let $v$ be the unit eigenvector of the claw $K_{1,3}$ (whose Laplacian is $L_{11}$ for the case $a_1=a_2=a_3=0$)
corresponding to the eigenvalue $4$. 
$v$ has the form $v=[\alpha,\beta,\beta,\beta]^{\mathsf{T}}$ where $\alpha<0$ and $\beta>0$. 
By the Courant-Fischer min-max characterization of 
eigenvalues~\cite[Theorem~8.1.2]{GOLUB-VANLOAN} 
we have 
\begin{align*}
\lambda_{\max}&\geq \lambda_{\max}(L_{11})\\
&\geq v^{\mathsf{T}}L_{11}v\\
&=[\alpha,\beta,\beta,\beta]\left( \mbox{\footnotesize $\begin{bmatrix}
  3&-1&-1&-1\\
  -1&1&0&0\\
  -1&0&1&0\\
  -1&0&0&1\\
\end{bmatrix}
+
\begin{bmatrix}
  0&0&0&0\\
  0&a_1&0&0\\
  0&0&a_2&0\\
  0&0&0&a_3\\
\end{bmatrix}$}\right)
\begin{bmatrix}
\alpha\\\beta\\\beta\\\beta  
\end{bmatrix}\\
&=4+\beta^2(a_1+a_2+a_3)\\
&\geq 4, 
\end{align*}
with the last equality holding if and only if 
$a_1=a_2=a_3=0$, that is, when 
the starlike tree is the claw $K_{1,3}$. 
\end{hide}
%The second statement about the necessary and sufficient condition on the 
%equality requires the argument in Section~\ref{sec:mK13},
%%equality requires the argument in Section~5, 
%in particular, Corollary~\ref{cor:mK13}.  From this, we can easily
%see that the only starlike tree having an eigenvalue exactly equal to 4 is
%$K_{1,3}$.

To prove the second statement about the necessary condition on the 
equality (the sufficiency is easily verified), first note that by 
\eqref{eqn:lowerbound}, $d(v_1) = 3$ is a necessary condition for 
$\lambda_{\max}=4$.  Let $d_1$ denote the highest degree of such a starlike tree,
i.e., $d_1=d(v_1)$.  Since we only consider starlike trees,
the second highest degree $d_2$ must be either 2 or 1.
Now, we use the following
\begin{theorem}[Das 2004, \cite{DAS-UPBDS}]
Let $G = (V, E)$ be a connected graph and $d_1 \neq d_2$ where
$d_1$ and $d_2$ are the highest and the second highest degree, respectively.
Then, $\lambda_{\max}(G) = d_1 + d_2$ if and only if $G$ is a star graph.
\end{theorem}
If $\lambda_{\max}(G) = 4$ and $d_1=3$, then using this theorem, we must have 
$d_2=1$.  Hence, $G$ must be a star graph with $d_1=3$ and $d_2=1$, i.e., $G=K_{1,3}$.
\end{proof}

As for the concentration/localization of the eigenvector $\bphi_{n-1}$
corresponding to the largest eigenvalue $\lambda_{n-1}$, 
we prove the following 
\begin{theorem}
\label{thm:concentration}
Let $\bphi_{n-1} = \(\phi_{1,n-1}, \cdots, \phi_{n,n-1}\)^\mathsf{T}$,
where $\phi_{j,n-1}$ is the value of the eigenvector corresponding to
the largest eigenvalue $\lambda_{n-1}$ at the vertex $v_j$,
$j=1, \ldots, n$.  Then, the absolute value of this eigenvector at
the central vertex $v_1$ cannot be exceeded by those at the
other vertices, i.e.,
\bdm
| \phi_{1,n-1} | > | \phi_{j,n-1} | , \quad j=2, \ldots, n .
\edm
\end{theorem}
To prove this theorem, we use the following lemma, which is simply a corollary
of Gerschgorin's theorem~\cite[Theorem~1.1]{VARGA}:
\begin{lemma}
\label{lem:gerschgorin-cor}
Let $A$ be a square matrix of size $n \times n$, 
$\lambda_k(A)$ be any eigenvalue of $A$, and 
$\bphi_k = (\phi_{1,k}, \ldots, \phi_{n,k})^{\mathsf{T}}$ be the corresponding
eigenvector.  Let $k^*$ denote the index of the largest eigenvector component
in $\bphi_k$, i.e.,
$\left| \phi_{k^*,k} \right| = \max_{j \in N} \left| \phi_{j,k} \right|$ where
$N := \{1,\ldots,n\}$.
Then, we must have $\lambda_k(A)\in \Gamma_{k^*}(A)$, where  
$\Gamma_i (A) := \left\{z \in \C: \left|z-a_{ii}\right| \leq \sum_{j \in N \setminus \{i\}}|a_{ij}| \right\}$ is the $i$th Gerschgorin disk of $A$.
In other words, for the index of the largest eigenvector component,
the corresponding Gerschgorin disk must contain the eigenvalue.
\end{lemma}

\begin{proof}
Recall the proof of Gerschgorin's theorem. 
The $k^*$th row of $A\bphi_k=\lambda _k \bphi_k$ yields 
\bdm
\left|\lambda _k -a_{k^* k^*}\right| 
\leq \sum_{j \in N \setminus \{k^*\}}
\left|a_{k^* j}\right| \frac{\left|\phi_{j,k}\right|}{\left|\phi_{k^*,k}\right|}\,\leq \sum_{j \in N \setminus \{k^*\}}\left|a_{k^*j}\right|. 
\label{base}
\edm
This implies $\lambda_k \in \Gamma_{k^*}(A)$, which proves the lemma.
\end{proof}

\begin{proof}[Proof of Theorem~\ref{thm:concentration}]
First of all, by Corollary~\ref{cor:starlike} we have $\lambda_{n-1} \geq 4$.
However, $\lambda_{n-1}=4$ happens only for $K_{1,3}$. In that case,
it is easy to see that this theorem holds by directly examining the eigenvector 
$\bphi_{n-1} = \bphi_{3} \propto (3, -1, -1, -1)^{\mathsf{T}}$.
Hence, let us examine the case $\lambda_{n-1} > 4$.  In this case, 
Lemma~\ref{lem:gerschgorin-cor} indicates 
$4 < \lambda_{n-1} \in \Gamma_{(n-1)^*}(L)$ where 
$(n-1)^* \in N$ is the index of the largest component in $\bphi_{n-1}$.
Now, note that the disk $\Gamma_{i}(L)$ for any vertex $v_i$ that has degree 
2 is $\{z \in \C: |z-2| \leq 2 \}$ (and $\{ z \in \C: |z-1| \leq 1\}$ for a degree 1 vertex). 
This means that the Gerschgorin disk $\Gamma_{(n-1)^*}$ containing 
the eigenvalue $\lambda_{n-1}>4$ cannot be in the union of the Gerschgorin disks
corresponding to the vertices whose degrees are 2 or lower.  
Hence the index of the largest eigenvector component in $\bphi_{n-1}$ must 
correspond to an index for which the vertex has degree 3 or higher.
In our starlike-tree case, there is only one such vertex, $v_1$, i.e.,
$(n-1)^*=1$.
\end{proof}

For different proofs without using Gerschgorin's theorem,
see Das \cite[Lemma~4.2]{DAS} and E.\ Woei's dissertation~\cite{WOEI-PHD}.
We note that our proof using Gerschgorin's disks is more powerful than
those other proofs and can be used for more general situations than
the starlike trees as we will see in Section~\ref{sec:genres}.

\begin{remark}
Let $\bphi=(\phi_1, \phi_2, \dots, \phi_n)^{\mathsf T}$ be an eigenvector of 
a starlike tree $S(n_1, \ldots, n_k)$ corresponding to the 
Laplacian eigenvalue $\lambda$.
Without loss of generality, 
let $v_2, \dots, v_{n_1+1}$ be the $n_1$ vertices along a branch emanating from 
the central vertex $v_1$ with $v_{n_1+1}$ being the leaf (or pendant) vertex.
Then, along this branch, the eigenvector components satisfy the following
equations:
\begin{eqnarray}
\lambda \phi_{n_1+1} &=& \phi_{n_1+1} - \phi_{n_1}, \label{eqn:leaf}\\
\lambda \phi_{j} &=& 2\phi_{j} - \phi_{j-1} - \phi_{j+1} 
\quad 2 \leq j \leq n_1. \label{eqn:middle}
\end{eqnarray}
From Eq.~\eqref{eqn:middle}, we have the following recursion relation:
\bdm
\label{eqn:recur}
\phi_{j+1} + (\lambda-2) \phi_{j} + \phi_{j-1} = 0, \quad
j = 2, \dots, n_1.
\edm
This recursion can be explicitly solved using the roots of the
characteristic equation
\begin{equation}
\label{eqn:char}
r^2 + (\lambda-2)r + 1 = 0,
\end{equation}
and when \eqref{eqn:char} has distinct roots $r_1, r_2$, 
the general solution can be written as
\begin{equation}
\label{eqn:gensol}
\phi_j = A r_1^{j-2} + B r_2^{j-2}, \quad j=2, \dots, n_1+1,
\end{equation}
where $A, B$ are appropriate
constants derived from the boundary condition \eqref{eqn:leaf}.
Now, let us consider these roots of \eqref{eqn:char} in detail.
%The determinant of \eqref{eqn:char} is
The discriminant of \eqref{eqn:char} is
\bdm
\label{eqn:det}
\D(\lambda) := (\lambda-2)^2 - 4 = \lambda (\lambda - 4) .
\edm
Since we know that $\lambda \geq 0$, this discriminant changes its
sign depending on $\lambda < 4$ or $\lambda > 4$.
(Note that $\lambda = 4$ occurs only for the claw $K_{1,3}$ 
on which we explicitly know everything; 
hence we will not discuss this case further in this remark.)
If $\lambda < 4$, then $\D(\lambda) < 0$ and it is easy to show that
the roots are complex valued with magnitude 1.
This implies that \eqref{eqn:gensol} becomes
\bdm
\label{eqn:solless4}
\phi_j = A' \cos(\omega(j-2)) + B' \sin(\omega(j-2)), \quad j=2, \dots, n_1+1,
\edm
where $\omega$ satisfies $\tan \omega = \sqrt{\lambda (4-\lambda)}/(2-\lambda)$,
and $A', B'$ are appropriate constants.
In other words, if $\lambda < 4$, the eigenvector along this branch is
of oscillatory nature.
On the other hand, if $\lambda > 4$, then $\D(\lambda) > 0$ and it is easy to 
show that both $r_1$ and $r_2$ are real valued with $-1 < r_1 = \(2-\lambda+\sqrt{\lambda(\lambda-4)}\)/2 < 0$ while $r_2 = \(2-\lambda-\sqrt{\lambda(\lambda-4)}\)/2< -1$.
%It is clear that the dominating part in \eqref{eqn:gensol} is the term $Br_2^{j-2}$, which grows exponentially with $j$.  
On the surface, the term $Br_2^{j-2}$ looks like a 
dominating part in \eqref{eqn:gensol}; however, we see from \eqref{eqn:leaf} 
that $|\phi_{n_1}|>|\phi_{n_1+1}|$, which means the real dominating part 
in \eqref{eqn:gensol} for $j=2,\ldots,n_1+1$ is the term $Ar_1^{j-2}$.
Hence we conclude that $|\phi_j|$ decays exponentially with $j$, that is, 
the eigenvector component decays rapidly towards the leaves. 
The siuation is the same for the other branches.
\end{remark}

In summary, we have shown that a starlike tree has only one eigenvalue
$\geq 4$, and its eigenvector is localized at the central vertex 
in the sense of Theorem 3.1. Furthermore,
 the other eigenvectors are of oscillatory nature. 
Therefore the phase transition phenomenon for a starlike tree
(with the eigenvalue 4 as its threshold) is completely understood.

\section{The Localization Phenomena on General Graphs}
\label{sec:genres}
Unfortunately, actual dendritic trees are not exactly starlike.
However, our numerical computations and data analysis on totally 179 RGCs
indicate that:
\bdm
\label{eqn:starlike-stat}
0 \leq \frac{ \#\{j \in N \cond d(v_j) > 2\} - m_G([4,\infty))  }{n}
\leq 0.047
\edm
for each RGC.
Hence, we can define the \emph{starlikeliness} $S\ell(T)$ of a given tree $T$ as
\bdm
\label{eqn:starlikeliness}
S\ell(T) := 
1 - \frac{ \#\{j \in N \cond d(v_j) > 2\} - m_T([4,\infty))  }{n}
\edm
We note that $S\ell(T) \equiv 1$ for a certain class of RGCs whose dendrites 
are sparsely spread (see \cite{SAITO-WOEI-DENDRITES} for the 
characterization).  This means that dendrites in that class are all close to 
a starlike tree or a concatenation of several starlike trees.
We show some examples of dendritic trees with 
$S\ell(T) \equiv 1$ and with $S\ell(T) < 1$ in Figure~\ref{fig:zoom}.
\begin{figure}
\begin{center}
\renewcommand{\subfigtopskip}{0pt}
\renewcommand{\subfigbottomskip}{0pt}
\renewcommand{\subfigcapskip}{0pt}
\renewcommand{\subfigcapmargin}{0pt}
\renewcommand{\thesubfigure}{(a)}
\subfigure[RGC \#100; $S\ell(T) \equiv 1$]{\includegraphics[width=0.49\textwidth]{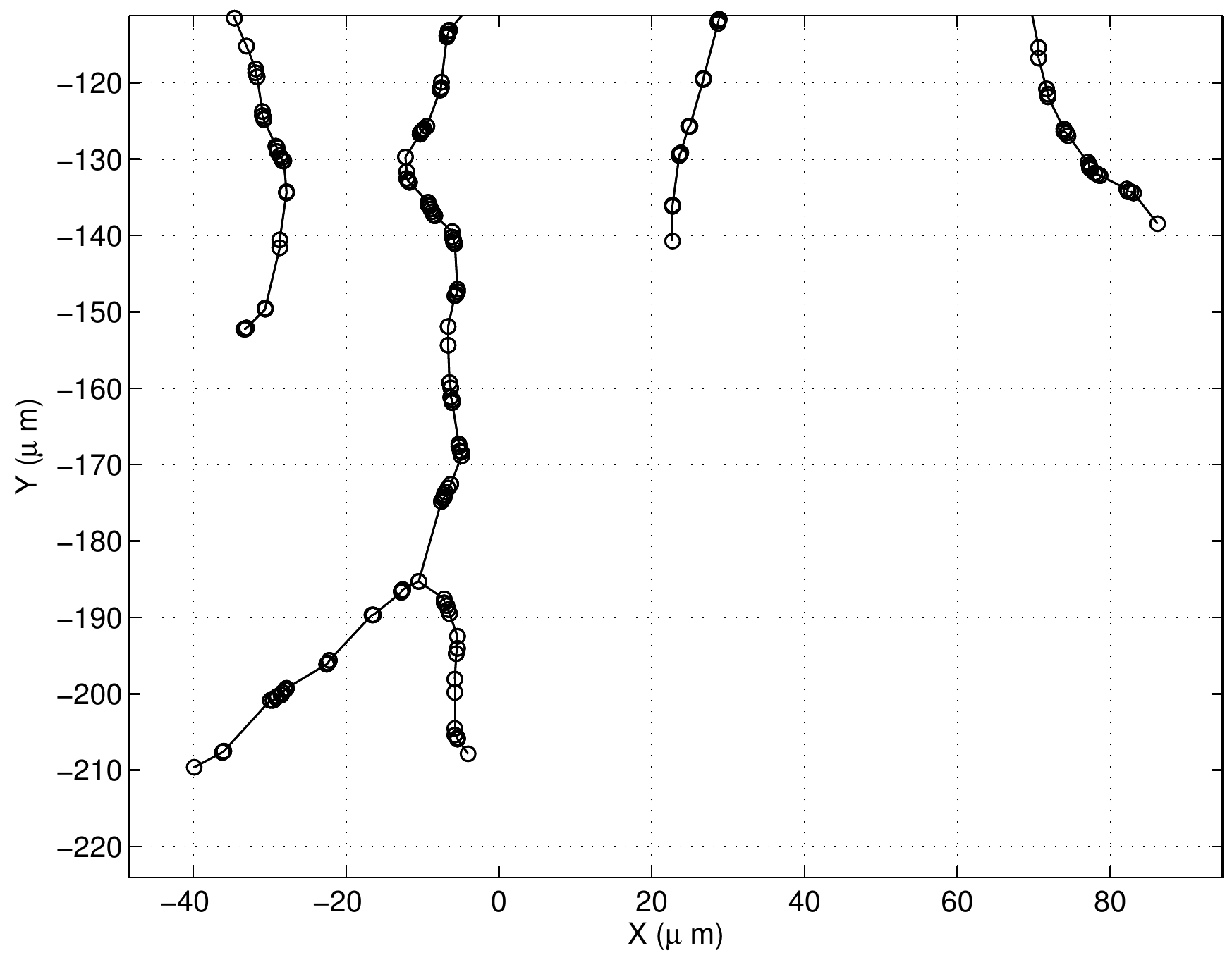}}
\renewcommand{\thesubfigure}{(b)}
\subfigure[RGC \#155; $S\ell(T)=0.953 < 1$]{\includegraphics[width=0.49\textwidth]{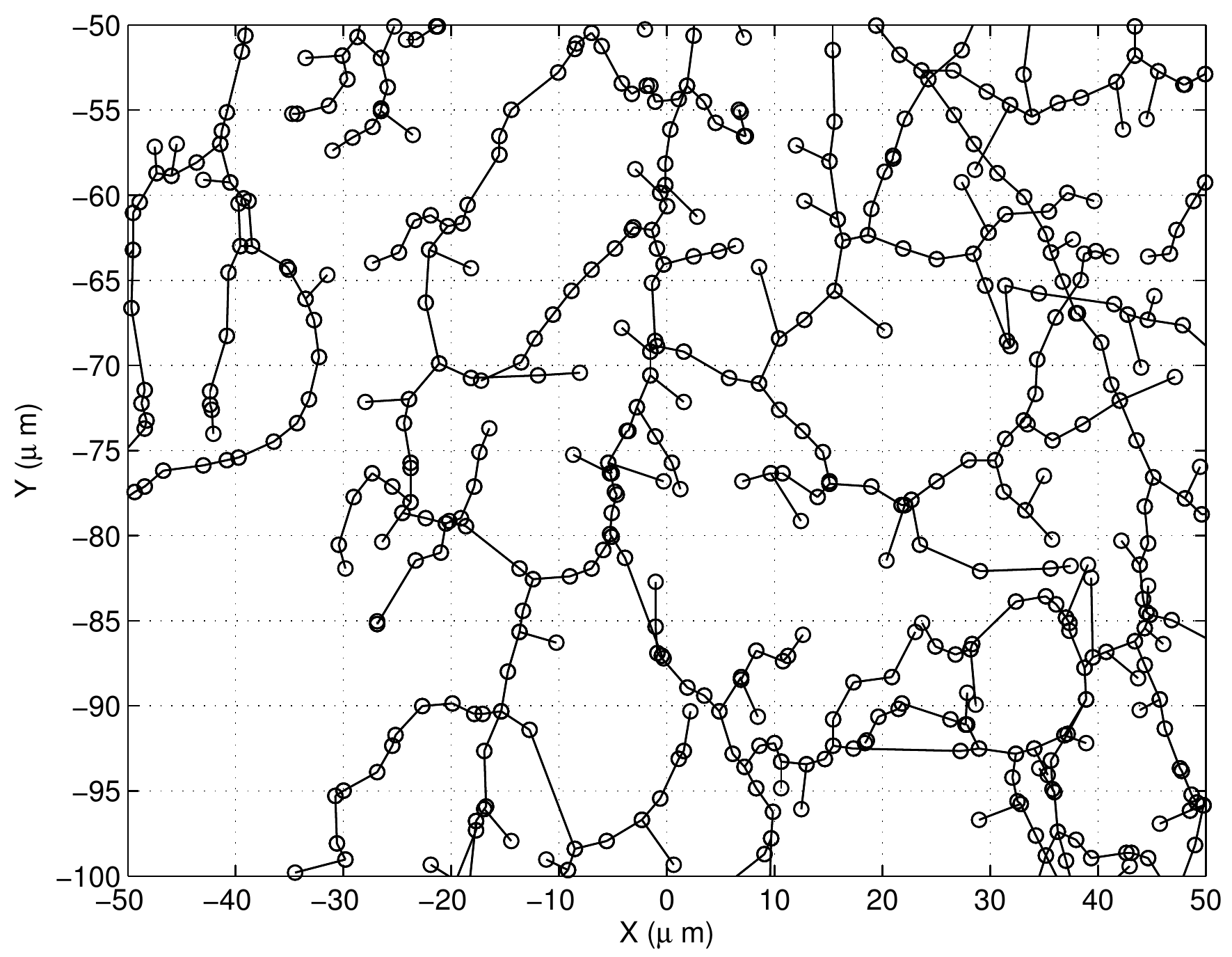}}
\end{center}
\caption{Zoomed-up versions of parts of some dendritic trees.}
\label{fig:zoom}
\end{figure}

The above observation has led us to prove the following
\begin{theorem}
\label{thm:ew4ltdeg2}
For any graph $G$ of finite volume, i.e., $\sum_{j=1}^n d(v_j) < \infty$,
we have
\bdm
0 \leq  m_G([4,\infty))  \leq \#\{j \in N \cond d(v_j) > 2\} 
\edm
and each eigenvector corresponding to $\lambda \geq 4$ has its
largest component (in absolute value) on the vertex 
whose degree is higher than 2.  
\end{theorem}
We refer the interested readers to \cite[Sec.~2]{MERRIS1} that reviews various
relationships between the multiplicity of certain eigenvalues and the 
graph structural properties different from our Theorem~\ref{thm:ew4ltdeg2}.
\begin{proof}
The  second statement follows from Lemma~\ref{lem:gerschgorin-cor}, 
because the Gerschgorin disks corresponding to 
vertices of degree 1 or 2 do not include $\lambda > 4$. 

We next prove the first statement. 
Let $L$ be a Laplacian matrix of $G$. We can apply a permutation $P$ such that 
\begin{equation}  \label{eq:Aperm}
P^{\mathsf{T}}LP=
\begin{bmatrix}L_1 & E^{\mathsf{T}}\\E  & L_2\end{bmatrix}, 
\end{equation}
where the diagonals of $L_1$ are  3 or larger (correspond to vertices of degree
$>2$), and the diagonals of $L_2$ are 2 or 1. 
Suppose $L_2$ is $\ell$-by-$\ell$. 
By Gerschgorin's theorem all the eigenvalues of $L_2$ must be 4 or below. 

In fact, we can prove the eigenvalues of $L_2$ are strictly below 4. 
By \cite[Theorem~1.12]{VARGA}, 
for an irreducible matrix (a Laplacian of a connected graph is irreducible) 
an eigenvalue can exist on the boundary of the union of the Gerschgorin disks only if it is the boundary of \emph{all} the disks. Furthermore, if there is such an eigenvalue, then the corresponding eigenvector has the property that all its components have the same absolute value.

Suppose on the contrary that $L_2\bx=4\bx$. Suppose without loss of generality that $L_2$ is irreducible; if not, we can apply a permutation so that $PL_2P^{\mathsf{T}}$ is block diagonal and treat each block separately. 
%For $L_2$ to have an eigenvalue equal to 4, 

Now if $L_2$ has a diagonal 1, then the corresponding Gerschgorin 
 disk lies on $[0,2]$, which does not pass 4. Hence by \cite[Theorem~1.12]{VARGA} this case is ruled out.
It follows that all the diagonals of $L_2$ are 2, and the sum of the absolute values of the rows of $L_2$ are all 4 (this happens only if $L_2$ is disjoint from $L_1$). So we need $\ell\geq 3$ (for example when $\ell=3$, 
$L_2=\mbox{\scriptsize
$\begin{bmatrix}
  2& -1& -1\\-1 &  2& -1\\ -1& -1 & 2
\end{bmatrix}$}$).
Now the $i$th ($1\leq i\leq \ell$) row of $L_2\bx=4\bx$ and the fact $|x_1|=|x_2|=\cdots =|x_\ell|$ force $x_i=-x_j$ for all $j\neq i$. This needs to hold for all $i$, which clearly cannot happen for $\ell\geq 3$. Therefore the eigenvalues of $L_2$ must be strictly below 4.

By the min-max characterization of the eigenvalues of $P^{\mathsf{T}}LP$, 
denoting by $\lambda_{\ell}(P^{\mathsf{T}}LP)$ the $\ell$th smallest eigenvalue,
we have 
\begin{align*}
\lambda_{\ell}(P^{\mathsf{T}}LP)&=\min_{\dim{S}=\ell}\max_{\by\in \mbox{\scriptsize span\normalsize}{(S)},\|\by\|_2=1}\by^{\mathsf{T}}(P^{\mathsf{T}}LP)\by. 
\end{align*}
Hence letting $S_0$ be the last $\ell$ column vectors of the identity $I_n$ and
noting $S_0^{\mathsf{T}}P^{\mathsf{T}}LPS_0=L_2$, we have
\begin{eqnarray*}
\lambda_{\ell}(P^{\mathsf{T}}LP) \leq \max_{\by\in S_0,\|\by\|_2=1}\by^{\mathsf{T}}(P^{\mathsf{T}}LP)\by &=& \lambda_{\mathrm{max}}(S_0^{\mathsf{T}}P^{\mathsf{T}}LPS_0) \\
&=& \lambda_{\mathrm{max}}(L_2).
\end{eqnarray*}
Since $\lambda_{\mathrm{max}}(L_2)<4$, we conclude that $P^{\mathsf{T}}LP$ (and hence
$L$) has at least $\ell$ eigenvalues smaller than 4, i.e.,
$m_G([0,4)) \geq \ell$.  Hence, $m_G([4,\infty)) = n - m_G([0,4)) \leq n - \ell
= \#\{j \in N \cond d(v_j) > 2\}$, which proves the first statement. 
\end{proof}

To give a further explanation for the eigenvector localization behavior 
observed in Introduction, we next show that eigenvector components of 
$\lambda>4$ must decay exponentially along a branching path. 
\begin{theorem}\label{thm:expdecay}
Suppose that a graph $G$ has a branch that consists of a path of length $k$, 
whose indices are $\{i_1,i_2,\ldots,i_k\}$ where $i_1$ is connected to the rest
of the graph and $i_k$ is the leaf of that branch.  Then for any eigenvalue 
$\lambda$ greater than 4, the corresponding eigenvector 
$\bphi=\(\phi_{1}, \cdots, \phi_{n}\)^\mathsf{T}$ satisfies
\begin{equation}
\label{eq:phidecay}
 |\phi_{i_{j+1}}| \leq \gamma|\phi_{i_j}| \quad \text{for $j=1,2,\ldots,k-1$},
\end{equation}
where 
\begin{equation}
\label{eq:decayrate}
\gamma := \frac{2}{\lambda-2} < 1.
\end{equation}
Hence $|\phi_{i_j}| \leq \gamma^{j-1}|\phi_{i_1}|$ for $j=1,\ldots,k$, that is,
the magnitude of the components of an eigenvector corresponding to 
any $\lambda > 4$ along such a branch decays exponentially toward
its leaf with rate at least $\gamma$. 
\end{theorem}
\begin{proof}
There exists a permutation $P$ such that 
\[\widehat L:= P^{\mathsf{T}}LP=
\begin{bmatrix}L_1 & E^{\mathsf{T}}\\ E & L_2\end{bmatrix}, 
\]
where% \small
\[L_2=
\begin{bmatrix}
 2 & -1 &        &        &        &    \\
-1 &  2 &   -1   &        &        &    \\
   & -1 &    2   &   -1   &        &    \\
   &    & \ddots & \ddots & \ddots &    \\
   &    &        &    -1  &    2   & -1 \\
   &    &        &        &   -1   &  1
  \end{bmatrix}
\in \Rf^{k \times k}
\]%\normalsize
and $E$ has a -1 in the top-right corner and 0 elsewhere.
The diagonals of $L_2$ correspond to the vertices $v_{i_1}, \ldots, v_{i_k}$ 
of the branch under consideration.

Let $L\bphi=\lambda \bphi$ with $\lambda>4$. 
We have $\widehat{L}\by=\lambda \by$ where 
$\by=(y_1, y_2,\cdots,y_n)^\mathsf{T}=P^{\mathsf{T}}\bphi$. 
Note that $(y_{n-k+1}, y_{n-k+2},\cdots,y_n)=(\phi_{i_1},\phi_{i_2},\ldots, \phi_{i_k})$. 
The last row of $\widehat{L}\by=\lambda \by$ gives
\[-y_{n-1}+y_{n}=\lambda y_n,\]
hence 
\begin{equation}
\label{eq:n-1}
|y_{n}|= \frac{1}{\lambda-1}|y_{n-1}| \leq \gamma |y_{n-1}|.  
\end{equation}
The $(n-1)$st row of $\widehat{L} \by=\lambda \by$ gives
\[-y_{n-2}+2y_{n-1}-y_{n}=\lambda y_{n-1}.\]
Using $|y_n| \leq |y_{n-1}|$ we get 
\begin{equation}
\label{eq:n-2}
|y_{n-1}|=\frac{|y_{n-2}+y_{n}|}{\lambda-2} \leq \frac{|y_{n-2}|+|y_{n-1}|}{\lambda-2},  
\end{equation}
from which we get $|y_{n-1}| \leq |y_{n-2}|$. Therefore $|y_n| \leq |y_{n-1}| \leq |y_{n-2}|$, and so
\[|y_{n-1}|=\frac{|y_{n-2}+y_{n}|}{\lambda-2} \leq \frac{2|y_{n-2}|}{\lambda-2}=\gamma|y_{n-2}|. \]
Repeating this argument $k-1$ times we obtain \eqref{eq:phidecay}. 
\end{proof}

We note that the inequalities \eqref{eq:n-1} and \eqref{eq:n-2} include 
considerable overestimates, and tighter bounds can be obtained at the cost of
simplicity. Hence in practice the decay rate is much smaller than $\gamma$
defined in \eqref{eq:decayrate}.  We also note that the larger the
eigenvalue $\lambda > 4$, the smaller the decay rate $\gamma$ is, i.e., the
faster the amplitude decays along the branching path.

Also note that the above result holds for any branching path of a tree. 
In particular, if a tree has $k$ branches consisting of paths, they must all
have the exponential decay in eigenvector components if $\lambda > 4$.  
This gives a partial explanation for the eigenvector localization behavior 
observed in Introduction.   However, the theorem cannot compare the 
eigenvector components corresponding to branches emanating from different 
vertices of degrees higher than 2, so a complete explanation remains an 
open problem.

\begin{remark}
\label{rem:ew4}
Let us briefly consider the case $\lambda=4$. In this case
we have $\gamma=\frac{2}{\lambda-2}=1$, suggesting the
corresponding eigenvector components along a branching path may not decay.
However, we can still prove that unless
$\phi_{i_1}=\phi_{i_2}=\cdots =\phi_{i_k}=0$, we must have
\begin{equation}  \label{eq:philam4}
|\phi_{i_k}| < |\phi_{i_{k-1}}|<\cdots <|\phi_{i_1}|.
\end{equation}
In other words, the eigenvector components must decay
along the branch, although not necessarily exponentially. To
see this, we first note that if $y_n=0$, then the last row
of $\widehat L \by = \lambda \by$ forces $y_{n-1}=0$. Then,
$y_n=y_{n-1}=0$ together with the $(n-1)$st row gives
$y_{n-2}=0$. Repeating this argument we conclude that $y_j$
must be zero for all $j=n-k+1,\ldots,n$. Now suppose that
$|y_n|>0$. Following the above arguments we see that the
inequality in \eqref{eq:n-1} with $\gamma=1$ must be strict, that is,
$|y_n|<|y_{n-1}|$. Using this we see that the inequality in
\eqref{eq:n-2} must also be strict, hence
$|y_{n-1}|<|y_{n-2}|$. Repeating this argument proves
\eqref{eq:philam4}.
\end{remark}

\section{A Class of Trees Having the Eigenvalue 4}
\label{sec:mK13}
As raised in Introduction, we are interested in answering
Q3: Is there any tree that possesses an eigenvalue exactly equal to 4?
To answer this question, we use the following result of Guo
\cite{GUO} (written in our own notation).
\begin{theorem}[Guo 2006, \cite{GUO}]
\label{thm:guo}
Let $T$ be a tree with $n$ vertices. Then,
\bdm
\lambda_j(T) \leq \left\lceil \frac{n}{n-j} \right\rceil, \quad j=0, \ldots, n-1,
\edm
and the equality holds %iff
if and only if all of the following hold:
 a) $j \neq 0$; b) $n-j$ divides $n$; and c)
$T$ is spanned by $n-j$ vertex disjoint copies of $K_{1,\frac{j}{n-j}}$.
\end{theorem}
\noindent
Here, a tree $T=T(V,E)$ is said to be spanned by $\ell$ 
%\emph{vertex disjoint copies} of $K$ if $K=K_i(V_i,E_i)$ for $i=1,\ldots,\ell$
%($K_i$'s are of course identical), $V=\bigcup_{i=1}^\ell V_i$, 
%and $V_i \cap V_j=\emptyset$ for all $i \neq j$.
\emph{vertex disjoint copies} of identical graphs $K_i(V_i,E_i)$ for 
$i=1,\ldots,\ell$ if $V=\bigcup_{i=1}^\ell V_i$ and 
$V_i \cap V_j=\emptyset$ for all $i \neq j$.
Figure~\ref{fig:mK13new}(a) shows an example of such vertex disjoint copies for 
$K_i = K_{1,3}$ by connecting their central vertices.
We note that there are many other ways to form disjoint vertex copies of $K_{1,3}$.

This theorem implies the following 
\begin{corollary}
\label{cor:mK13}
A tree has an eigenvalue exactly equal to 4 
if 
%$\leftarrow$ if and only if 
it is spanned by
$m (= n/4 \in \N)$ vertex disjoint copies of $K_{1,3} \equiv S(1,1,1)$.
\end{corollary}

%\begin{proof}
%Setting $n/(n-j) = 4$ implies $3n=4j$.  Since 3 and 4 are relatively prime,
%there exists $m \in \N$ such that $n=4m$ and $j=3m$.  
%Hence Guo's theorem with $n=4m$ and $j=3m$ guarantees that
%the eigenvalue exactly equal to 4 occurs at $j=3m$, i.e., $\lambda_{3m} = 4$,
%iff the tree is spanned by $m$ vertex disjoint copies of $K_{1,3}$.
%\end{proof}
\begin{figure}
\renewcommand{\subfigtopskip}{0pt}
\renewcommand{\subfigbottomskip}{0pt}
\renewcommand{\subfigcapskip}{0pt}
\renewcommand{\subfigcapmargin}{0pt}
\begin{center}
\renewcommand{\thesubfigure}{(a)}
%\subfigure[$\ $]{\includegraphics[width=0.275\textwidth]{mK13.pdf}}
\subfigure[$\ $]{\includegraphics[width=0.5\textwidth]{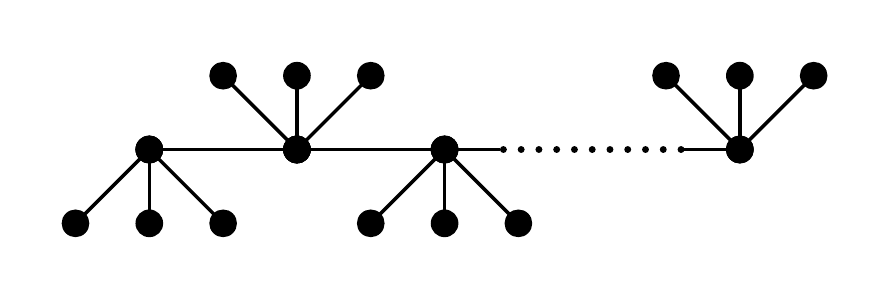}}
%\vspace{2mm}
\renewcommand{\thesubfigure}{(b)}
\subfigure[$\ $]{\includegraphics[width=0.475\textwidth]{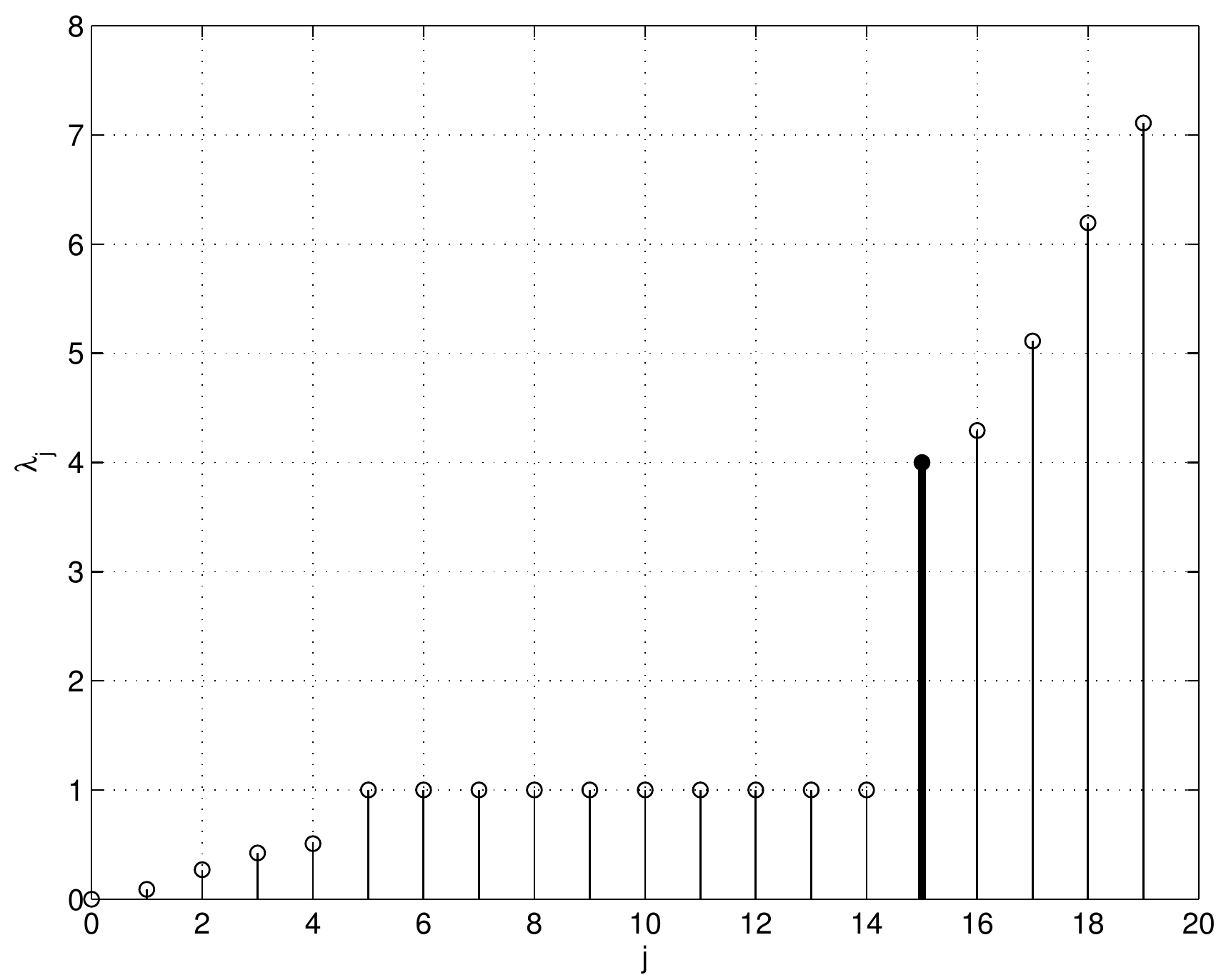}}
\caption{(a) A tree spanned by multiple copies of $K_{1,3}$ connected via their
central vertices.  This tree has an eigenvalue equal to 4 with multiplicity 1.
(b) The eigenvalue distribution of such a tree spanned by 5 copies of
$K_{1,3}$. We note that $S\ell(T)=1$ for this tree. }
\label{fig:mK13new}
\end{center}
\end{figure}

%Figure~\ref{fig:mK13new}(a) illustrates such an example while
Figure~\ref{fig:mK13new}(b) shows the eigenvalue distribution of a tree 
spanned by $m=5$ copies of $K_{1,3}$ as shown in Figure~\ref{fig:mK13new}(a). 
Regardless of $m$, the eigenvector 
corresponding to the eigenvalue 4 has only two values: one constant value at the
central vertices, and the other constant value of the opposite sign at the 
leaves, as shown in Figure~\ref{fig:mK13}(a). 
By contrast, the eigenvector corresponding to the largest eigenvalue is again 
concentrated around the central vertex as shown in Figure~\ref{fig:mK13}(b).
\begin{figure}
\renewcommand{\subfigtopskip}{0pt}
\renewcommand{\subfigbottomskip}{0pt}
\renewcommand{\subfigcapskip}{0pt}
\renewcommand{\subfigcapmargin}{0pt}
\begin{center}
\renewcommand{\thesubfigure}{(a)}
\subfigure[$\bphi_{15}$]{\includegraphics[width=0.49\textwidth]{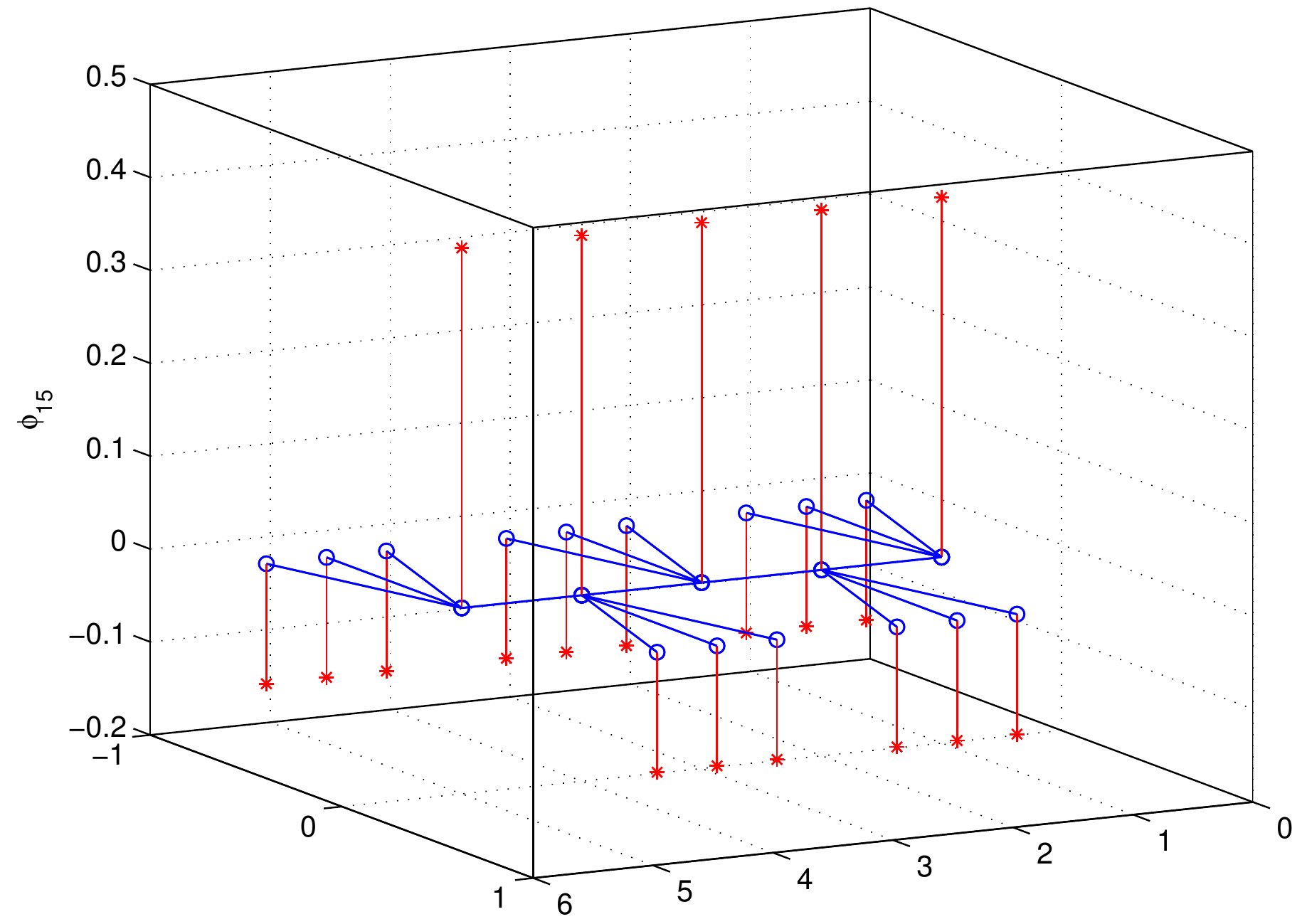}}
\renewcommand{\thesubfigure}{(b)}
\subfigure[$\bphi_{19}$]{\includegraphics[width=0.49\textwidth]{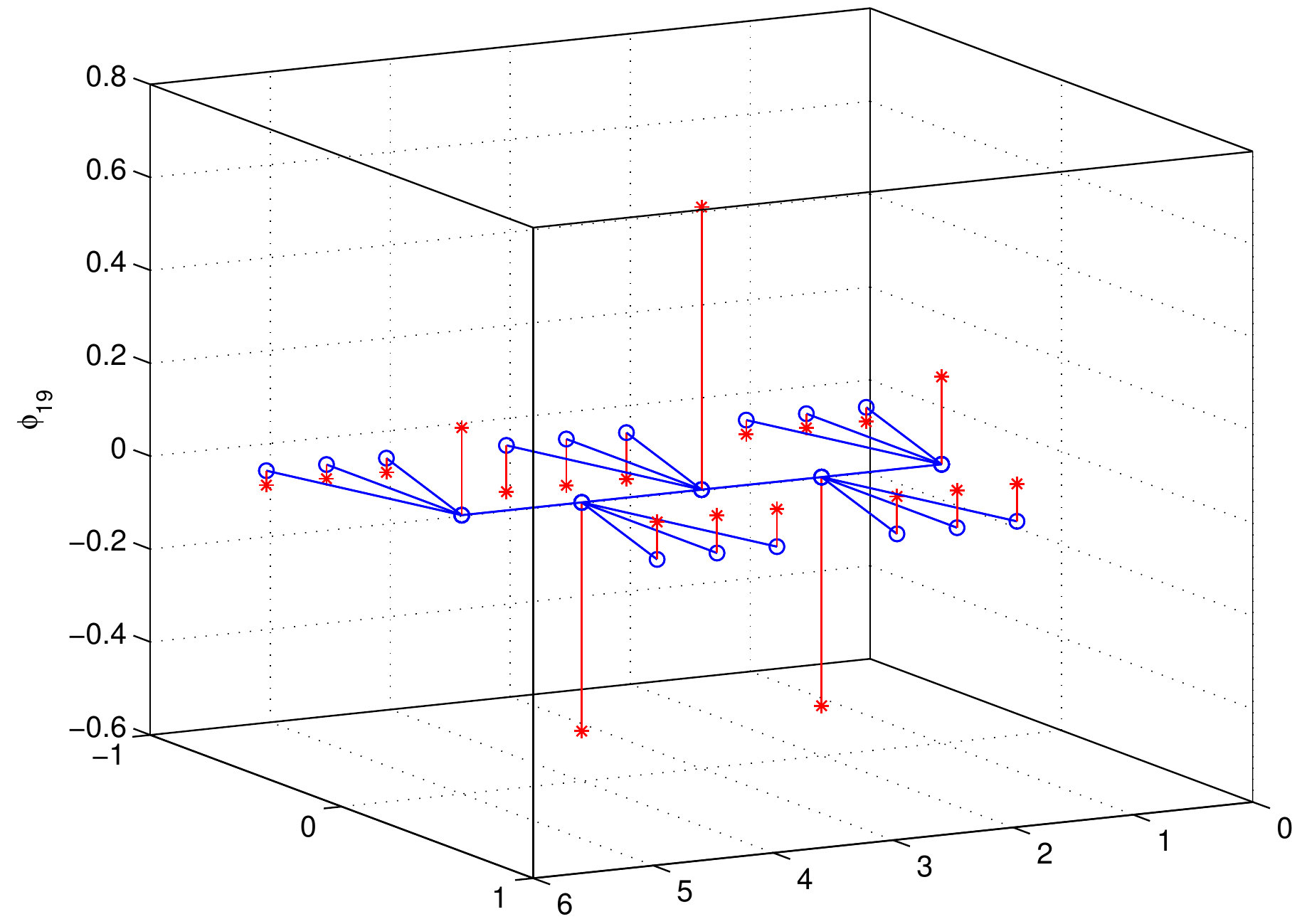}}
\caption{
(a) The eigenvector $\bphi_{15}$ corresponding to $\lambda_{15}=4$ in the 3D perspective view. 
(b) The eigenvector $\bphi_{19}$ corresponding to the maximum eigenvalue 
$\lambda_{19}=7.1091$, which concentrates around the central vertex.}
\label{fig:mK13}
\end{center}
\end{figure}

%\begin{remark}
Theorem~\ref{thm:guo} asserts that a general tree $T$ with $n$ vertices can have
\emph{at most} $\lfloor n/4 \rfloor$ Laplacian eigenvalues $\geq 4$.
We also know by Theorem~2.1 of \cite{GRONE-MERRIS-SUNDER} that any tree $T$
possessing the eigenvalue $4$ must have multiplicity $m_T(4) = 1$ and $n=4m$ for some 
$m \in \N$. 
%Together with Corollary~\ref{cor:mK13}, it also asserts that $m$ vertex disjoint
Hence, Theorem~\ref{thm:guo} also asserts that trees spanned by $m$ vertex
disjoint copies of $K_{1,3}$ form the only class of trees for which 4 is the
$3n/4 (=3m)$th eigenvalue of $T$; in other words, trees in this class are 
the only ones that have \emph{exactly} $n/4 (=m)$ eigenvalues $\geq 4$.  

Trees spanned by vertex disjoint copies of $K_{1,3}$, however, are not the only
ones that have an eigenvalue exactly equal to 4. 
For instance, Example~2.9 of \cite{GRONE-MERRIS-SUNDER}, which has $n=36=4 \cdot 9$ vertices and is called $Z_4$ as shown in Figure~\ref{fig:z4}, 
is \emph{non-isomorphic} to any tree spanned by 9 vertex disjoint copies 
of $K_{1,3}$; 
yet it has $\lambda_{30}=4$ (but $\lambda_{27} = 1 \neq 4$).

On the other hand, we have the following
\begin{prop}
\label{thm:n8}
If $n\leq 11$, any tree possessing an eigenvalue exactly equal to 4 must
be spanned by vertex disjoint copies of $K_{1,3}$.
\end{prop}
\begin{proof}
First of all, 4 must divide $n$, hence $n=4m$ with $m=1$ or $m=2$.
If $m=1$, then we know from Theorem~\ref{thm:guo} that 
$\lambda_{max} = \lambda_3 \leq 4$ and $\lambda_2 \leq 2$.
Hence, $\lambda_3 = 4$ is the only possibility, and consequently 
$T=K_{1,3}$ using the same theorem.
If $m=2$, then Theorem~\ref{thm:guo} states that
$\lambda_{max} = \lambda_7 \leq 8$; $\lambda_6 \leq 4$; and $\lambda_5 \leq 3$; 
\dots
If $\lambda_6 = 4$, then the necessary and sufficient conditions for the 
equality in Theorem~\ref{thm:guo} state that $T$ must be spanned by
two vertex disjoint copies of $K_{1,3}$, so we are done.
Now we still need to show that $\lambda_7$ cannot be 4.
Let $d_1$ be the degree of the highest degree vertex of a tree $T$ 
under consideration.
Then, we have
\begin{equation}
\label{eqn:lbd-ubd}
d_1 + 1 \leq \lambda_7 < d_1 + 2\sqrt{d_1 - 1}.
\end{equation}
Here, the lower bound is due to Grone and Merris \cite{GRONE-MERRIS-2}
and the upper bound is due to Stevanovi\'c \cite{STEVANOVIC}.
Now, we need to check a few cases of the values of $d_1$.
\begin{itemize}
\item If $d_1 = 2$, then the upper bound in \eqref{eqn:lbd-ubd} is 4.  
Hence, $\lambda_7 = 4$ cannot happen.
(This includes the case of a path graph that cannot reach the eigenvalue 4).
So, we must have $d_1 \geq 3$.
\item If $d_1 > 3$, of course, the lower bound in \eqref{eqn:lbd-ubd} is greater
than 4. Hence, $\lambda_7$ cannot be 4 either.
\item Finally, if $d_1 = 3$, then the above bounds are: $4 \leq \lambda_7 < 5.8284 \cdots$.  Can $\lambda_7 = 4$ in this case?  According to Zhang and Luo \cite{ZHANG-LUO}, 
the equality in that lower bound holds if and only if there exists a vertex that
is adjacent to all the other vertices in $T$.  That is, the degree of that 
vertex is $n-1=7$.  Since $d_1=3$, this cannot happen.
\end{itemize}
Hence, for $n=8$, the only possibility for a tree $T$ to have an eigenvalue
exactly equal to 4 is the case when $\lambda_6 = 4$, which happens
if and only if $T$ is spanned by two vertex disjoint copies of $K_{1,3}$.
\end{proof}
It turns out, however, that proving the necessity for $n > 11$ using similar arguments quickly becomes cumbersome, even for the next step $n=12$. 
At this point, we do not know whether there are other classes of trees than
$Z_4$ discussed above or those spanned by vertex disjoint copies of $K_{1,3}$
that can have an eigenvalue exactly equal to 4.
Hence, identifying every possible tree that has an eigenvalue exactly equal to 4 is an open problem. 
%we would like to pose the following challenge (or an open question)
%to the interested reader: 
%\emph{``Identify every possible tree that has an eigenvalue exactly equal to 4.''}
%Hence, 
%%\end{remark}
\begin{figure}
\begin{center}
\includegraphics{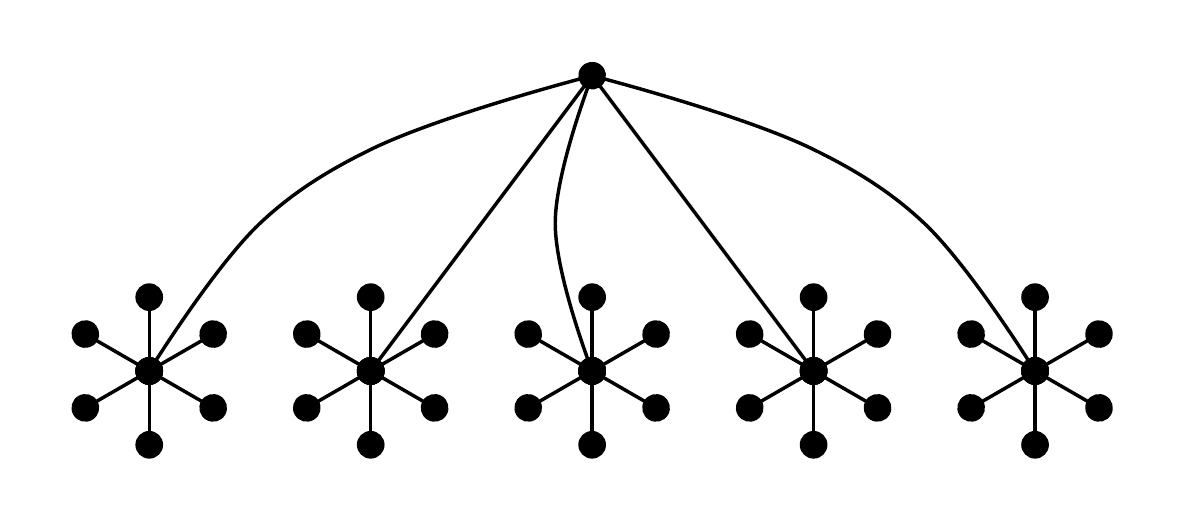}
\caption{Yet another tree $Z_4$ (Example~2.9 of \cite{GRONE-MERRIS-SUNDER} with
$k=4$) that has an eigenvalue exactly equal to 4.  This is non-isomorphic
to any tree spanned by vertex disjoint copies of $K_{1,3}$ such as the one shown
in Figure~\ref{fig:mK13new}(a).}
\label{fig:z4}
\end{center}
\end{figure}

\section{Implication of a Long Path on Eigenvalues}
\label{sec:eigconsequence}
In Section~\ref{sec:genres} we saw that for a graph that has a branch consisting
of a long path, its Laplacian eigenvalue greater than 4 has the property that 
the corresponding eigenvector components along the branch must decay exponentially. 

Here we discuss a consequence of such a structure in terms of the eigenvalues. 
We consider a graph $G$ formed by connecting two graphs $G_1$ and $G_3$ 
with a path $G_2$. 
%, see Figure~\ref{fig:g1g2g3} for an illustration. 
Note that this is a more general graph than in Section~\ref{sec:genres} 
%Note that this is a more general graph than in Section~4
(which can be regarded as the case without $G_3$). 
We show that if $G_2$ is a long path then any eigenvalue greater than 4 of the 
Laplacian of either of the two subgraphs $G_1 \cup G_2$ and $G_2 \cup G_3$ must
be nearly the same as an eigenvalue of the Laplacian of the whole graph $G$. 

\begin{theorem}
\label{thm:g1g2g3}
Let $G$ be a graph %as in Figure~\ref{fig:g1g2g3} 
obtained by connecting two graphs with a path, whose Laplacian $L$
can be expressed as 
\[L=\begin{bmatrix}
L_1 & E_1^{\mathsf{T}}&0\\
E_1 & L_2&E_2^{\mathsf{T}}\\
0 &E_2&L_3\\
\end{bmatrix},\]
where $E_1$ and $E_2$ have -1 in the top-right corner and 0 elsewhere. 
$L_i$ is $\ell_i\times \ell_i$ for $i=1,2,3$ and $L_2$ represents the path $G_2$,
that is, a tridiagonal matrix with 2 on the diagonals and -1 on the 
off-diagonals.   

Let $\widetilde{\lambda}>4$ be any eigenvalue of the top-left 
$(\ell_1+\ell_2)\times (\ell_1+\ell_2)$ 
(or bottom-right $(\ell_2+\ell_3)\times (\ell_2+\ell_3)$) submatrix of $L$. 
Then there exists an eigenvalue $\lambda$ of $L$ such that 
\begin{equation}  \label{eq:lam4mu}
|\lambda-\widetilde{\lambda}|\leq \widetilde{\gamma}^{\ell_2}, 
\end{equation}
where $\widetilde{\gamma}:=\frac{2}{\widetilde{\lambda}-2}<1$. 
\end{theorem}
\begin{proof}
We treat the case where $\widetilde{\lambda}$ is an eigenvalue of the top-left
$(\ell_1+\ell_2)\times (\ell_1+\ell_2)$ part of $L$, which we denote by $L_{12}$.
The other case is analogous. 

As in Theorem~\ref{thm:expdecay}, we can show that any 
eigenvalue $\widetilde{\lambda}>4$ of $L_{12}$ has its corresponding eigenvector
components decay exponentially along the path $G_2$. This means 
that the bottom eigenvector component is smaller than 
$\widetilde{\gamma}^{\ell_2}$ in absolute value (we normalize the eigenvector 
so that it has unit norm) where 
$\widetilde{\gamma}:=\frac{2}{\widetilde{\lambda}-2}<1$ as in \eqref{eq:decayrate}. 

Let $L_{12}=Q \Lambda Q^{\mathsf{T}}$ be an eigendecomposition where 
$Q^{\mathsf{T}}Q=I$ and the eigenvalues are arranged so that 
$\widetilde{\lambda}$ appears in the top diagonal of $\Lambda$. 
For notational convenience let $\ell_{12} := \ell_1+\ell_2$.
Then, consider the matrix
\begin{equation}
  \label{eq:Lrelation}
\widehat L=\begin{bmatrix}
Q^{\mathsf{T}}&0 \\  0&I
\end{bmatrix}L\begin{bmatrix}
Q&0 \\  
0&I
\end{bmatrix}=\begin{bmatrix}
\Lambda & \bv \bbe_1^{\mathsf{T}}\\
\bbe_1\bv^{\mathsf{T}} &L_3\end{bmatrix},
\end{equation}
where $\bbe_1=(1,0,\ldots,0)^{\mathsf{T}} \in \Rf^{\ell_3}$ and 
$\bv=(v_1,\ldots,v_{\ell_{12}})^{\mathsf{T}} \in \Rf^{\ell_{12}}$. Direct calculations
show that $v_i=-q_{\ell_{12},i}$ where $q_{\ell_{12},i}$ is the bottom component of the
eigenvector $\bq_i$ of $L_{12}$ corresponding to the $i$th eigenvalue. 
In particular, by the above argument we have 
$|q_{\ell_{12},1}|=|v_{1}|\leq \widetilde{\gamma}^{\ell_2} (\ll 1)$.

Note that in the first row and column of $\widehat L$, the only nonzeros are the
diagonal (which is $\widetilde{\lambda}$), and the $(1,\ell_{12}+1)$ and 
$(\ell_{12}+1,1)$ entries, both of which are equal to $v_{1}$. 
Now, viewing the $(1,\ell_{12}+1)$ and $(\ell_{12}+1,1)$ entries of 
$\widehat L$ as perturbations (write $\widehat L=\widehat L_1+\widehat L_2$ 
where $\widehat L_1$ is obtained by setting the $(1,\ell_{12}+1)$ and 
$(\ell_{12}+1,1)$ entries of $\widehat L$ to 0) and 
using Weyl's theorem \cite[Theorem~8.1.5]{GOLUB-VANLOAN} we see that there 
exists an eigenvalue $\lambda$ of $\widehat L$ (and hence of $L$) that lies in 
the interval $[\widetilde{\lambda}-\|\widehat L_2\|_2,\widetilde{\lambda}+\|\widehat L_2\|_2]=
[\widetilde{\lambda}-|v_{1}|,\widetilde{\lambda}+|v_{1}|]$. 
Together with $|v_{1}|\leq \widetilde{\gamma}^{\ell_2}$ we obtain 
\eqref{eq:lam4mu}. 
\end{proof}

Recall that $\widetilde{\gamma}^{\ell_2}$ decays exponentially with $\ell_2$, and 
it can be negligibly small for moderate $\ell_2$; for example, for 
$(\lambda,\ell_2)=(5,30)$ we have $\widetilde{\gamma}^{\ell_2}=5.2\times 10^{-6}$. 
We conclude that the existence of a subgraph consisting of a long path implies
that the eigenvalues $\lambda>4$ of a subgraph must match those of the whole 
graph to high accuracy.

\section{On the Eigenvector of the Largest Eigenvalue}
\label{sec:counterex}
In view of the results in Section~\ref{sec:genres} it is natural to ask whether
it is always true that the largest component of the eigenvector corresponding to
the largest eigenvalue of a Laplacian matrix of a graph lies on the vertex
of the highest degree. 
Here we show by a counterexample that this is not necessarily true. 

Consider for example a tree as in Figure~\ref{fig:mK13-ga}, which is generated
as follows: first we connect $m$ copies of $K_{1,2}$ (equal to $P_3$) as shown in 
Figure~\ref{fig:mK13new}(a); then add to the right a comet $S(\ell,1,1,1,1)$ as 
in Figure~\ref{fig:starlike}(b). 

\begin{figure}
  \begin{center}
    \includegraphics[width=0.9\textwidth]{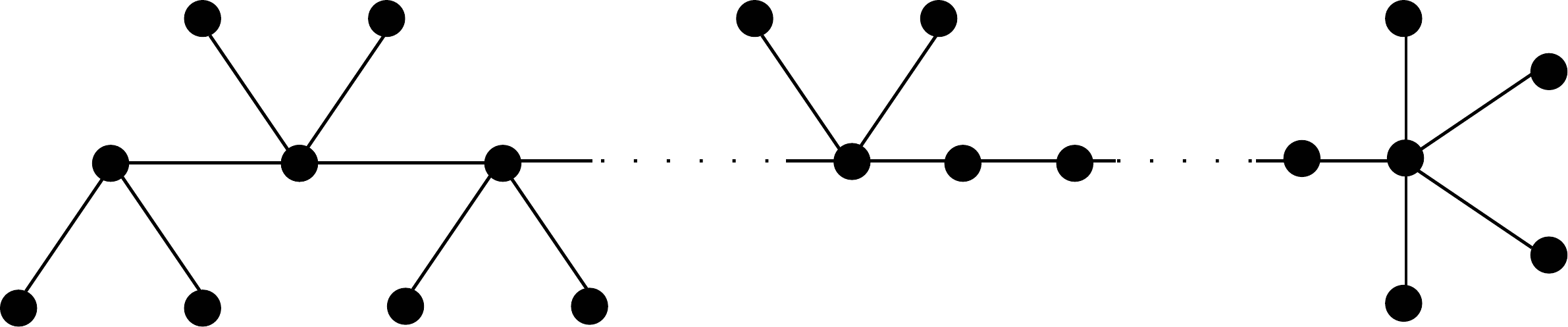}
\caption{Counterexample graph for the conjecture.}
\label{fig:mK13-ga}
  \end{center}
\end{figure}

Now for sufficiently large $m$ and $\ell$ ($m,\ell\geq 5$ is sufficient), 
the largest component in the eigenvector $\bphi$ corresponding to the largest
eigenvalue of the resulting Laplacian $L$ occurs at one of the central vertices
of $K_{1,2}$, not at the vertex of degree 5 belonging to the comet.

Let us explain how we came up with this counterexample. 
The idea is based on two facts. 
The first is the discussion in Section~\ref{sec:eigconsequence}, where we noted
that a long path $G_2$ implies any eigenvalue larger than 4 must be close to an
eigenvalue of a subgraph $G_1 \cup G_2$ or $G_2 \cup G_3$. 
Therefore, in the notation of Section~\ref{sec:eigconsequence}, by connecting
two graphs ($G_1=mK_{1,2}$ and $G_3=K_{1,5}$, a star) 
with a path $G_2$ such that 
the largest eigenvalue $\widetilde{\lambda}$ of $L_{12}$ is larger than that of
$L_{23}$, we ensure that the largest eigenvalue $\lambda$ of $L$ is very close to
$\widetilde{\lambda}$. 
The second is the Davis-Kahan $\sin\theta$ theorem \cite{DAVIS-KAHAN}, which
states that a small perturbation of size $\widetilde{\gamma}^{\ell_2}$ in the 
matrix $\widehat L_1$ (recall the proof of Theorem~\ref{thm:g1g2g3})
can only induce small perturbation also in the eigenvector: its angular 
perturbation is bounded by $\widetilde{\gamma}^{\ell_2}/\delta$, where $\delta$ is
the distance between $\lambda$ and the eigenvalues 
of $\widehat L$ after removing its first row and column. 
Furthermore, the eigenpair $(\widetilde{\lambda},\widetilde{\bphi})$ of 
$\widehat L_1$ satisfies $\widetilde{\bphi}=(1,0,\ldots,0)^{\mathsf{T}}$, and 
the eigenvectors $\widehat \bphi$ ($\simeq \widetilde{\bphi}$ by Davis-Kahan) of
$\widehat L$ and $\bphi$ of $L$ corresponding to $\lambda$ are related by 
$\bphi=\left[\begin{smallmatrix}Q&0\\0&I\end{smallmatrix}\right]\widehat \bphi$,
which follows from \eqref{eq:Lrelation}. 
Therefore, $\bphi$ has its large components at the vertices belonging to $G_1$. 
In view of these our approach was to find two graphs $G_1$ and $G_3$ such that
the highest degrees of the vertices of $G_1$ and $G_3$ are 4 and 5, respectively,
and the largest eigenvalue of the Laplacian of $G_1$ is larger than that of $G_3$.

\section{Discussion}
\label{sec:disc}
In this paper, we obtained precise understanding of the phase transition
phenomenon of the combinatorial graph Laplacian eigenvalues and eigenvectors 
for starlike trees.  For a more complicated class of graphs including
those representing dendritic trees of RGCs, we proved in 
Theorem~\ref{thm:ew4ltdeg2} that the number of the eigenvalues greater than or
equal to 4 is bounded from above by the number of vertices whose degrees are 
strictly higher than 2.  In Theorem~\ref{thm:expdecay}, we proved
that if a graph has a branching path, the magnitude of the components
of an eigenvector corresponding to any eigenvalue greater than 4
along such a branching path decays exponentially toward its leaf.
In Remark~\ref{rem:ew4}, we also extended Theorem~\ref{thm:expdecay}
for the case of $\lambda=4$ although the decay may not be exponential.
%We also answered Q3 raised in Introduction by proving Corollary~\ref{cor:mK13}.
%In other words, we identified a special class of trees spanned by 
%vertex disjoint copies of the claw $K_{1,3}$, which is the only class of trees
%that can have an eigenvalue exactly equal to 4.

%We saw that a class of trees exists that have an eigenvalue 4. 
As for Q3 raised in Introduction---``Is there any tree that possesses an eigenvalue exactly equal to 4?''---we showed that any tree with $n$ vertices
($n=4m$ for some $m \in \N$) spanned by $m$ vertex disjoint copies of $K_{1,3}$
possesses an eigenvalue exactly equal to 4 in Corollary~\ref{cor:mK13}, and
that such class of trees are the only ones that can have an eigenvalue exactly 
equal to 4 if $n\leq 11$ 
%$n < 8$  (i.e., $m=1$ or $m=2$)} 
in Proposition~\ref{thm:n8}.
On the other hand, for larger $n$,
%$n > 8$,
 we pointed out that not only those spanned by vertex disjoint copies of $K_{1,3}$, but also a tree called $Z_4$ discovered
in \cite{GRONE-MERRIS-SUNDER} and shown in Figure~\ref{fig:z4}
have an eigenvalue exactly equal to 4.  
A challenging yet interesting 
question is whether or not one can identify every possible tree that has
an eigenvalue 4.

Another quite interesting question is Q4 raised in Introduction: ``Can a 
simple and connected graph, not necessarily a tree, have eigenvalues 
equal to 4?''  The answer is a clear ``Yes.''  For example, 
the $d$-cube ($d > 1$), i.e., the $d$-fold Cartesian product of $K_2$ 
with itself is known to have the Laplacian eigenvalue 4 with multiplicity $d(d-1)/2$;
see e.g., \cite[Sec.~4.3.1]{BIYIKOGLU-LEYDOLD-STADLER}. 

Another interesting example is a regular finite lattice graph in 
$\Rf^d$, $d > 1$, which is simply the $d$-fold Cartesian product of 
a path $P_n$ shown in Figure~\ref{fig:path} with itself.
Such a lattice graph has repeated eigenvalue 4.  
In fact, each eigenvalue and the corresponding eigenvector of such a lattice 
%graph of size $n \times n \times \dots \times n = n^d$ can be written as
graph can be written as
\begin{eqnarray}
\lambda_{j_1, \dots, j_d} &=& 4 \sum_{i=1}^d \sin^2\( \frac{j_i \pi}{2n} \) 
\label{eqn:lattice-ew}\\
\phi_{j_1, \dots, j_d}(x_1, \dots, x_d) &=& \prod_{i=1}^d \cos \( \frac{j_i \pi (x_i+\frac{1}{2})}{n} \),
\label{eqn:lattice-ev}
\end{eqnarray}
where $j_i, x_i \in \Z/n\Z$ for each $i$,
as shown by Burden and Hedstrom \cite{BURDEN-HEDSTROM}.
Note that \eqref{eqn:lattice-ew} and \eqref{eqn:lattice-ev} are also valid for
$d=1$.  In that case these reduce to \eqref{eqn:1d-ew} 
%and \eqref{eqn:1d-ev} 
that we already examined in Section~\ref{sec:defs}.

Now, determining $m_G(4)$, i.e., the multiplicity of the eigenvalue 4 of
this lattice graph, is equivalent to finding the number of the integer solutions
$(j_1, \dots, j_d) \in \(\Z/n\Z\)^d$ 
to the following equation:
%\vspace{-3mm}
\begin{equation}
\label{eqn:lattice-sine-sum}
\sum_{i=1}^d \sin^2\( \frac{j_i \pi}{2n} \) = 1.
\end{equation}
For $d=1$, there is no solution as we mentioned in Section~\ref{sec:defs}.
For $d=2$, it is easy to show that $m_G(4) = n-1$ by direct examination
of \eqref{eqn:lattice-sine-sum} using some trigonometric identities.
For $d=3$, $m_G(4)$ behaves in a much more complicated manner, which is
deeply related to number theory. We expect that more complicated situations
occur for $d > 3$.
We are currently investigating this 
%interesting problem
 on regular finite
lattices.
On the other hand, it is clear from \eqref{eqn:lattice-ev} that 
the eigenvectors corresponding to the eigenvalues greater than or equal to
4 on such lattice graphs cannot be localized or concentrated on those vertices
whose degree is higher than 2 unlike the tree case.
Theorem~\ref{thm:expdecay} and Remark~\ref{rem:ew4} do not apply either
since such a finite lattice graph do not have branching paths.

Finally, we would like to note that even a simple path, such as the one shown in
Figure~\ref{fig:path}, exhibits the eigenfunction localization phenomena 
\emph{if it has nonuniform edge weights}, which we recently observed numerically.
We will report our progress on investigation of localization phenomena on such 
weighted graphs at a later date.

\section*{Acknowledgments}
We thank the referees for their remarks and suggestions. 
This research was partially supported by the following grants from the Office
of Naval Research: N00014-09-1-0041; N00014-09-1-0318.
A preliminary version of a part of the material in this paper 
\cite{SAITO-WOEI-KOKYUROKU} was presented at the workshop on 
``Recent development and scientific applications in wavelet analysis'' held 
at the Research Institute for Mathematical Sciences (RIMS), Kyoto University,
Japan, in October 2010, and at the 7th International Congress on Industrial and
Applied Mathematics (ICIAM), held in Vancouver, Canada, in July 2011.

%\begin{hide}
%\section*{References}

%\end{hide}

%\section*{References}
%\bibliographystyle{elsarticle-num}
%\bibliography{/h/saito/report/all}

\end{document}